\documentclass[12pt]{amsart}
\usepackage{amsmath,amsthm,amsfonts,amssymb,eucal}

\renewcommand {\a}{ \alpha }
\renewcommand{\b}{\beta}

\newcommand{\g}{\gamma}

\renewcommand{\d}{\delta}
\newcommand{\s}{\sigma}
\renewcommand{\l}{\lambda}
\renewcommand{\L}{\Lambda}
\newcommand{\z}{\zeta}
\renewcommand{\t}{\theta}

\newcommand{\p}{\partial}
\newcommand{\om}{\omega}
\newcommand{\Om}{\Omega}

\newcommand{\oq}{\ {\raise 7pt\hbox{${\scriptstyle\circ}$}}
\kern -7pt{
\hbox{$Q$}}}

\newcommand{\R}{ \mathbb R}

\newcommand{\Rd}{ \mathbb R^d}

\newcommand {\GS}{\mathfrak S}

\newcommand {\bx}{\mathbf x}

\newcommand {\bz}{\mathbf z}
\newcommand {\by}{\mathbf y}

\newcommand{\SN}{{\sf{N}}}

\newcommand {\bmu}{\boldsymbol\mu}

\newcommand {\boldeta}{\boldsymbol\eta}

\newcommand {\bxi}{\boldsymbol\xi}

\newcommand{\lu}{\langle}
\newcommand{\ru}{\rangle}

\newcommand{\CB}{\mathcal B}

\newcommand{\CH}{\mathcal H}

\newcommand{\plainC}[1]{\textup{{\textsf{C}}}^{#1}}

\newcommand{\plainL}[1]{\textup{{\textsf{L}}}^{#1}}

\DeclareMathOperator{\tr}{{tr}}

\newcommand{\1}
{{\,\vrule depth3pt height9pt}{\vrule depth3pt height9pt}
{\vrule depth3pt height9pt}{\vrule depth3pt height9pt}\,}

\DeclareMathOperator {\im }{{Im}}

\DeclareMathOperator {\dist} {{dist}}

\DeclareMathOperator{\op}{{Op}}

\newtheorem{thm}{Theorem}[section]
\newtheorem{cor}[thm]{Corollary}
\newtheorem{lem}[thm]{Lemma}
\newtheorem{prop}[thm]{Proposition}
\newtheorem{cond}[thm]{Condition}

\theoremstyle{definition}

\newtheorem*{remark}{Remark}
\newtheorem{rem}[thm]{Remark}

\numberwithin{equation}{section}

\newcommand{\bee}{\begin{equation}}
\newcommand{\ene}{\end{equation}}
\newcommand{\bees}{\begin{equation*}}
\newcommand{\enes}{\end{equation*}}
\newcommand{\bes}{\begin{split}}
\newcommand{\ens}{\end{split}}

\newcommand{\bet}{\begin{thm}}
\newcommand{\ent}{\end{thm}}
\newcommand{\bel}{\begin{lem}}
\newcommand{\enl}{\end{lem}}
\newcommand{\bec}{\begin{cor}}
\newcommand{\enc}{\end{cor}}
\newcommand{\bep}{\begin{proof}}
\newcommand{\enp}{\end{proof}}
\newcommand{\ber}{\begin{rem}}
\newcommand{\enr}{\end{rem}}

\usepackage{color,dsfont}
\begin{document}
\hoffset -4pc

\title
[Trace formulas for Wiener--Hopf operators]
{{Trace formulas for Wiener--Hopf 
operators with applications to entropies 
of free fermionic equilibrium states}}
\author[H.~Leschke, A.V.~Sobolev, W.~Spitzer]
{Hajo Leschke, Alexander V.~Sobolev, Wolfgang Spitzer}
\address{Institut f\"ur Theoretische Physik, 
Universit\"at Erlangen-N\"urnberg, 
Staudtstra\ss e 7, 91058 Erlangen, Germany}
\email{hajo.leschke@physik.uni-erlangen.de}
\address{Department of Mathematics\\ University College London\\
Gower Street\\ London\\ WC1E 6BT UK}
\email{a.sobolev@ucl.ac.uk}
\address{Fakult\"at f\"ur Mathematik und Informatik, 
FernUniversit\"at Hagen, Universit\"atsstra\ss e 1, 
58097 Hagen, Germany}
\email{wolfgang.spitzer@fernuni-hagen.de}
\keywords{Non-smooth functions of 
Wiener--Hopf operators, asymptotic 
trace formulas, entanglement entropy}
\subjclass[2010]{Primary 47G30, 35S05; Secondary 45M05, 47B10, 47B35}

\begin{abstract}
We consider 
non-smooth functions of (truncated) 
Wiener--Hopf type operators on the Hilbert space $\plainL2(\mathbb R^d)$. 
Our main results are uniform estimates for trace norms ($d\ge 1$) 
and quasiclassical asymptotic formulas for traces of the 
resulting operators ($d=1$). 
Here, we follow Harold Widom's seminal ideas, who proved 
such formulas for smooth functions decades 
ago. The extension to non-smooth functions 
and the uniformity of the estimates in various (physical) parameters
rest on recent advances by one of the authors (AVS). 
We use our results 
to obtain the large-scale behaviour of the local entropy and the spatially 
bipartite entanglement entropy (EE) of thermal equilibrium states of non-interacting 
fermions in position space $\mathbb R^d$ ($d\ge 1$) at positive temperature, $T>0$.  
In particular, our definition of the thermal EE 
 leads to estimates that are simultaneously sharp for small $T$
and large scaling parameter $\alpha>0$ provided that the
product $T\alpha$ remains bounded from below. Here $\alpha$
is the reciprocal quasiclassical parameter.
For $d=1$ we obtain for the  thermal EE an asymptotic formula which is 
consistent with the large-scale behaviour of the ground-state EE (at $T=0$),  
previously established by the authors for $d\ge 1$.
\end{abstract}

\date{May 14, 2016}

\maketitle

\tableofcontents

\section{Introduction} \label{intro}  

The present paper is devoted to the study of (bounded, self-adjoint) operators 
of the form
\begin{align}\label{WH:eq}
W_\a := W_\a(a; \L) := \chi_\L \op_\a(a) \chi_\L,\ \a >0,
\end{align}
on $\plainL2(\R^d)$, $d\ge 1$, 
where $\chi_\L$ is the indicator function of a set $\L\subset\R^d$.
The parameter $1/\a$ can be interpreted as a quasiclassical parameter
that tends to zero in our asymptotic results. 
The notation $\op_\a(a)$ stands for the $\a$-pseudo-differential 
operator with symbol $a=a(\bxi)$, which acts on Schwartz functions $u$ on $\R^d$ as
\begin{equation*}
\bigl(\op_\a(a) u\bigr)(\bx) := \frac{\a^{d}}{(2\pi)^{\frac{d}{2}}}
\iint e^{i\a\bxi\cdot(\bx-\by)} a(\bxi) u(\by) d\by d\bxi\,,\quad \bx\in\R^d.
\end{equation*} 
Integrals without indication of the 
integration domain always mean integration 
over $\R^d$ with the value of $d$ which is clear from the context.
More general symbols, depending on both variables 
$\bx$ and $\bxi$, or operators 
with matrix-valued symbols can be also treated, 
but we limit our attention only to $\bxi$-dependent symbols. 
We call the operator \eqref{WH:eq} a Wiener--Hopf operator. 
A more precise term would be \textit{truncated} 
Wiener--Hopf operator, but we always omit ``truncated" for brevity.  
Our focus is on the operator difference 
\begin{equation}\label{Dalpha:eq}
D_\a(a, \L; f) := \chi_\L f(W_\a(a; \L)) \chi_\L 
- W_\a(f\circ a; \L),
\end{equation}
with some suitably chosen functions $f$. We are interested in the 
asymptotic properties of the trace 
$\tr D_\a(a, \L; f)$ as $\a\to\infty$. 
If $f(0) = 0$, $\L$ is bounded and $a$ decays sufficiently fast 
at infinity, then it is trivial to observe that 
the second operator on the right-hand side of \eqref{Dalpha:eq} 
is trace-class and 
\begin{equation}\label{weyl:eq}
\tr W_\a(f\circ a; \L) = \frac{\a^d}{(2\pi)^d} |\L| \int f\bigl(a(\bxi)\bigr)d\bxi,
\end{equation} 
where $|\Lambda|$ is 
the $d$-dimensional Lebesgue measure of $\Lambda$.
If $|\L|=\infty$, then neither of the terms on the right-hand side 
of \eqref{Dalpha:eq} is trace class (except in trivial cases), 
but their difference is trace class, under the conditions adopted in this paper. 
We must emphasise that it is essential to us to consider $\L$ in \eqref{Dalpha:eq} 
of infinite measure.

Asymptotic properties of $D_\a(a, \L; f)$ 
have been extensively studied in the literature, 
with the majority of results obtained in the 1980's. All 
results obtained at that time pertained to the case of 
smooth functions $f$ (or more precisely, smooth on the range  
of the symbol $a$) and bounded $\L$. 
Under these assumptions, 
the case of a smooth symbol $a$ was 
understood particularly well: the full asymptotic expansion of 
$\tr D_\a(a, \L; f)$ in powers of 
$\a^{-1}$ was derived by 
A.~Budylin--V.~Buslaev \cite{BuBu} 
and  H.~Widom \cite{Widom_85}. 
The paper \cite{Widom_85} also 
provides a brief historical account of this problem. 
Out of all relevant bibliography we mention just one 
other paper by H.~Widom, \cite{Widom_82}, 
whose ideas we exploit in some of our proofs. 

Another important and challenging 
problem is to study the asymptotics of the trace of 
$D_\a(a, \L; f)$ for discontinuous symbols, in particular,
for symbols of the form $a = \chi_\Om$ with a bounded region $\Omega\subset\R^d$.
This problem was studied by H.~Landau--H.~Widom \cite{Land_Wid},  
H.~Widom \cite{Widom_821} (for $d=1$) 
and by A.V.~Sobolev \cite{Sob, Sob2} (for arbitrary $d\ge 1$). 
It was found that 
\begin{equation}\label{Widom_conj:eq}
\tr D_\a(a, \L; f) = \mathfrak W_1 \,\a^{d-1}\log(\a) + o(\a^{d-1}\log(\a))\,, \;\a\to\infty,
\end{equation}
for a bounded domain $\L\subset \R^d$ with an explicitly given coefficient 
$\mathfrak W_1 = \mathfrak W_1(\p\L, \p\Om, f)$. The discontinuity 
of the symbol $a$ can be interpreted 
as the presence of one of the 
two \textit{Fisher--Hartwig singularities} 
investigated in detail for 
truncated Toeplitz matrices, that is,
for the discrete counterpart 
of Wiener--Hopf operators, see \cite{DIK}.  

In recent years, new demands 
for the asymptotics of traces 
of Wiener--Hopf operators emerged, 
which have been triggered by applications 
to (quantum) statistical mechanics. 
Our interest originates from the 
large-scale behaviour of the spatially 
bipartite entanglement entropy 
(EE, also called mutual information) 
of free fermions in thermal equilibrium. 
Here one faces several mathematical challenges at the same time. 

\begin{enumerate}
\item \textbf{Non-smooth functions $f$.}
One needs to consider the operator 
\eqref{Dalpha:eq} with 
functions $f$ that lack smoothness at finitely many points,  
or, which is the same in view of additivity, at one point. 
The functions of interest are the 
\textit{$\g$-R\'enyi entropy functions} $\eta_\g, \g >0,$ 
that are defined 
in \eqref{eta_gamma:eq} and \eqref{eta1:eq}.

\item 
\textbf{Unbounded $\L$.}
One needs to consider the operator \eqref{Dalpha:eq} 
with unbounded domains $\L$, in contrast to most of 
the previously known results. 
\item 
\textbf{Uniform estimates.}
In quantum-mechanical applications, apart from the scaling parameter, 
it is natural to control the dependence of the symbol $a$ 
on other parameters such as the temperature $T\ge 0$. Thus it is necessary 
to provide estimates and asymptotic remainder estimates that are uniform in the 
symbol $a$ in some broad sense. 
For example, in the study of the entanglement entropy  
the symbol $a$ in the operator \eqref{Dalpha:eq} 
is given by the Fermi symbol $a_{T, \mu}$, see 
the definition \eqref{positiveT:eq}, and one 
needs to control the $T$-dependence of the estimates. 
%
%
%
This requires substantial extra work since the results of \cite{Widom_82, Widom_85} are not directly applicable.
\end{enumerate}

%
A general approach to the study of operator differences of the form 
$P f(PAP)P - Pf(A)P$ with a self-adjoint operator $A$, an orthogonal projection 
$P$ and a non-smooth function $f$, was put forward in \cite{Sob_14}. 
One application of the results in \cite{Sob_14} is the extension of \eqref{Widom_conj:eq}   
to non-smooth functions $f$ under the assumption that either $\L$ or its complement 
is bounded, thereby tackling challenges (1) and (2) above. 

The special case $a(\bxi) = \chi_\Om(\bxi)$ for bounded 
$\Om, \L\subset\R^d$ was considered even earlier in \cite{LeSpSo}. 
In the quantum-mechanical context, formula \eqref{Widom_conj:eq}, 
if used with $a = \chi_\Om$ and the function 
$f=\eta_\g$,  gives the large-scale asymptotics of 
the entanglement entropy at zero temperature with Fermi sea $\Om$, 
see also \cite{GK} for some other motivation.

In the present paper we work exclusively with 
smooth symbols $a$ with a fast decay at infinity. 
The function $f$ is allowed to lack smoothness at one point, see 
Condition \ref{f:cond}. 
A typical example of such a function is $f(t) = |t|^\g, \g >0$. 
The region $\L$ is such that either $\L$ or $\R^d\setminus\L$ is bounded, 
see Condition \ref{domain:cond} for details.   
 
The goal of this paper is two-fold, and it correspondingly 
splits in two parts. 

 \textbf{Part 1: Sections  
\ref{sect:estimates}--\ref{multiple:sect}.}   
First we establish some explicit estimates for the 
(quasi-) norms of the operator \eqref{Dalpha:eq} in the Schatten--von Neumann 
classes $\GS_q$, $q\in (0, 1]$. Later on we need only trace class norms, but 
the more general $\GS_q$-estimates are obtained at ``no extra cost", 
and are provided for the sake of completeness. Here we rely 
on the results of \cite{Sob_14} 
where this problem was studied in the abstract setting. We quote these results 
in Proposition \ref{Szego1:prop}. 
%
%
Indeed, the very fact that 
$D_\a(a, \L; f)\in \GS_q$ is an almost direct consequence 
of Proposition \ref{Szego1:prop}, 
but this alone is insufficient for us since we need sharp 
explicit estimates, uniform in $a$. 
Thus we identify a class of symbols $a$ that we call 
\textit{multi-scale symbols}, and establish explicit estimates 
for $\|D_\a(a, \L; f)\|_{\GS_q}$, 
which are uniform in some suitable sense, see Remark 
\ref{uniform:rem}. They do turn 
out to be sharp in $\a$ and $T$ 
when used for the symbol 
\eqref{positiveT:eq}, which serves as our leading example. The main estimate is contained in Theorem \ref{multi:thm}. This takes care of issue (3) above. 

Our next result is the asymptotic formula for $\tr D_\a(a, f; \L)$ 
as $\a\to\infty$, 
for spatial dimension $d = 1$, see Section \ref{sect:asympt}. 
Here we assume again that $a$ is a multi-scale symbol, and 
the main objective is to have the explicit 
control of the remainder, see 
Theorems \ref{interval_fn:thm} and 
\ref{scale_asymp:thm}.  
As mentioned earlier, we follow the seminal 
ideas of H.~Widom, who proved such asymptotic 
results for smooth functions $f$ already in the 1980's, see \cite{Widom_80,Widom_82,Widom_85,Widom_87}. 
The proofs of the main asymptotic results of Section \ref{sect:asympt} are presented in Sections \ref{proofs:sect} and \ref{res1:sect}. 
To accomplish this we use rather a standard methodology 
of quasiclassical analysis: first we prove the required 
asymptotics for smooth functions $f$, and then using the bounds 
from Theorem \ref{multi:thm} we extend them to non-smooth ones. 
The starting point 
is the Helffer--Sj\"ostrand formula (see Appendix A) 
which rewrites the trace of $D_\a(a,\L; f)$ for smooth $f$  
in terms of $D_\a(a,\L;r_z)$ with 
the resolvent function $r_z(\l) := (\l-z)^{-1},\l\in\R$, 
$z\in\mathbb C$. 

 \textbf{Part 2: Sections \ref{gen_bounds:sect}--\ref{entropy:sect}.} 
Here we apply the results obtained in Part 1 to the symbol 
\begin{equation}\label{positiveT:eq} 
a(\bxi) := a_{T, \mu}(\bxi) 
:= \frac{1}{1+ \exp\big(\frac{h(\bxi) - \mu}{T}\big)}\,,\quad \bxi\in\R^d,
\end{equation}
which is nothing but the Fermi symbol of free Fermions.
Here the real-valued 
function $h = h(\bxi)$ is the classical one-particle Hamiltonian of the free Fermi gas, 
$h(\bxi)\to\infty$ as $|\bxi|\to\infty$, 
the parameter $T>0$ is the (absolute) temperature, and 
$\mu \in\R$ is the chemical potential. 
We always assume that $\mu$ is fixed and $T\in (0, T_0]$ for some $T_0>0$. 
We are interested in the behaviour of 
$D_\a(T) = D_\a(a_{T, \mu}, \L; f)$ 
when $\a\to\infty$ and $T\downarrow 0$ simultaneously.  
The symbol $a_{T, \mu}$ fits in the formalism of 
multi-scale symbols, 
laid out in Section \ref{estim:sect}, and as a result we derive from  
Theorem \ref{multi:thm} a sharp estimate for the trace norm of 
$D_\a(T)$ with explicit dependence 
on $T$ and $\a$, under the condition $\a T\ge 1$, see Theorem \ref{entropy:thm}.  
For $d=1$ 
the sharpness of this estimate is confirmed by the asymptotic formulas
\eqref{fn_lowT:eq}, \eqref{fn_lowT1:eq} for the trace of $D_\a(T)$, that are 
derived from Theorem \ref{scale_asymp:thm}. 
The extension of this large-scale asymptotics to 
dimensions $d\ge2$ is the 
content of a separate paper \cite{Sob-future}. 
 
In Section \ref{entropy:sect} we specialise further to the 
function $f = \eta_\g, \g >0$, which brings us to 
the main application of our results, that is, to the large scale 
 asymptotic formulas 
 for the  entanglement entropy (EE) 
 $\mathrm{H}_\g(T, \mu; \a\L)$  
 of free fermions in thermal equilibrium 
 associated with the bipartition $\R^d = \L \cup \L^c$ 
 with a bounded $\L\subset\R^d$, at temperature $T>0$.  
 
As pointed out earlier, by now the EE is well-understood at 
zero temperature (see \cite{GK,LeSpSo}), which corresponds to the case when the Fermi symbol 
$a$ is given by the indicator function 
$\chi_\Om$ of the Fermi sea $\Om\subset\R^d$. 
In this case the EE exhibits a
logarithmically enhanced area-law scaling of the form 
\eqref{Widom_conj:eq}. 
The case $T>0$ is somewhat trickier:  the entropy 
of the total system on $\mathbb R^d$, that is,
$\tr \eta_\g(W_\a(a_{T, \mu};\mathbb R^d))
= \tr \op_\a(a_{T, \mu})$, 
is infinite, and hence it is not clear in advance 
even how to define the EE 
in a meaningful way. Intuitively, the EE measures
the difference between the sum of the 
entropies of the states localised to $\L$ and $\L^c$ and
the entropy of the total system. Therefore, a physically and mathematically
reasonable definition of the EE is given in \eqref{def:EE} below. By that we not only 
ensure the finiteness of the EE, but are also able to obtain sharp (in $\a$ and $T$) 
upper bounds in any spatial dimension $d\ge 1$. In Theorem 
\ref{EE_bound:thm} we show that 
$|\mathrm{H}_\g(T, \mu; \a\L)|\le C\a^{d-1}(|\log(T)|+1)$
\footnote{Here and everywhere below by $C$ or $c$, with or without indices, we denote 
positive, finite constants, whose exact values are unimportant.}, 
if $\a\ge1$, $\a T\ge 1$. This bound tallies well with the 
asymptotics \eqref{Widom_conj:eq}, and thus supports   
the intuitive expectation that the scaling behaviour of the EE at $T>0$ 
should resemble more and more the zero temperature behaviour,  
as $T\downarrow 0$. For $d=1$ this expectation is further justified 
by the asymptotic formulas \eqref{EEas:eq} and \eqref{tostep:eq}, 
derived from \eqref{fn_lowT:eq}, see Section \ref{low T:sect} for the low $T$-behaviour of 
the asymptotic coefficient. As a by-product, this leads to the two-term asymptotic 
expansion of the local thermal entropy of the free Fermi gas, 
which extends the hitherto known leading Weyl asymptotics (see \cite{PS,BHK}). 

The paper \cite{LeSpSo2} presents results on the EE for the
one-dimensional and the multi-dimensional case without the underlying mathematical details. 
In combination with \cite{Sob-future} and \cite{LeSoSp-future} the present 
paper provides then a full proof of these announcements.

\section{Estimates}\label{sect:estimates}
 
\subsection{The Schatten--von Neumann ideals of compact operators} 
This paper relies on the results obtained in \cite{Sob_14} for general quasi-normed ideals of compact operators. 
Here we limit our attention to the case of Schatten--von Neumann operator ideals $\GS_q, q>0$. 
Detailed information on these ideals can be found e.g. in \cite{BS,GK,Pie,Simon}. We shall point out only some basic facts.
For a compact operator $A$ on a separable Hilbert space $\CH$ denote 
by $s_n(A), n = 1, 2, \dots$ its singular values, that is, the eigenvalues of the operator $|A| := \sqrt{A^*A}$. 
We denote the identity operator on $\CH$ by $\mathds{1}$. The Schatten--von Neumann ideal $\GS_q, q>0$ consists of all 
compact operators $A$, for which 
\begin{equation*}
	\|A\|_{\GS_q} := \biggl[\sum_{k=1}^\infty s_k(A)^q\biggr]^{\frac{1}{q}}<\infty.
\end{equation*}
If $q\ge 1$, then the above functional defines a norm; if $0<q <1$, then it is a so-called quasi-norm. 
There is nevertheless a convenient analogue of the triangle inequality, which is called 
the \textit{$q$-triangle inequality:}
\begin{equation}\label{qtriangle:eq} 
\|A_1+A_2\|_{\GS_q}^q\le \|A_1\|_{\GS_q}^q + \|A_2\|_{\GS_q}^q, \ \ A_1, A_2\in\GS_q,\ 0<g\le 1,
\end{equation}
and the H\"older inequality,
\begin{equation}\label{Holder:eq}
\|A_1 A_2\|_{\GS_q}\le \|A_1\|_{\GS_{q_1}} \cdot \|A_2\|_{\GS_{q_2}}, \ \ q^{-1} = q_1^{-1} + q_2^{-1}\,,\ \ 0 < q_1,q_2\le\infty\,,
\end{equation} 
see \cite{Rot} and also \cite{BS}. In what follows we focus on the case $q\in (0, 1]$. 

\subsection{Non-smooth functions}
We study non-smooth functions, satisfying the following condition:

\begin{cond}\label{f:cond}
For some integer $n \ge 1$ the function $f\in\plainC{n}(\R\setminus\{ t_0 \})\cap\plainC{}(\R)$ satisfies the bound 
\begin{equation}\label{fnorm:eq}
\1 f\1_n := \max_{0\le k\le n}\sup_{t\not = t_0} |f^{(k)}(t)| |t-t_0|^{-\g+k}<\infty
\end{equation}
with some $\g >  0$, and is supported on the interval $(t_0-R, t_0+R)$ with some $R>0$.  

The case $R\ = \infty$ means no restriction on the support of the function $f$. 
\end{cond}

Below we denote by $\chi_R$ the indicator function of the interval $(-R, R)$, $R>0$.
For a function $f$ satisfying the above condition the following bound holds for $t\not = t_0$: 
\begin{equation}\label{fbound:eq}
|f^{(k)}(t)| \le \1 f\1_n |t - t_0|^{\g-k}\chi_R(t - t_0), \ 
k = 0, 1, \dots , n. 
\end{equation}
If $n\ge 1$, then the above condition implies that with 
$\varkappa := \min\{1, \g\}$ 
the function $f$ is $\varkappa$-H\"older continuous --- 
we denote this set by $\plainC{0, \varkappa}(\R)$. 
In particular, one can show that for any $t_1, t_2\in\R$,
\begin{equation}\label{hol:eq}
|f(t_1) - f(t_2)|\le 2 R^{\g-\varkappa} 
\1 f\1_1 |t_1-t_2|^\varkappa,\ \varkappa = \min\{1,\gamma\}.
\end{equation}
The following Proposition was proved in \cite{Sob_14}. For simplicity we state it only for bounded self-adjoint operators. 
 
\begin{prop}\label{Szego1:prop} 
Suppose that $f$ satisfies Condition \ref{f:cond} with some $\g >0$,  $n\ge 2$ and 
some $t_0\in\R$, $R\in (0, \infty)$. Let $q$ be a number such that $(n-\s)^{-1} < q\le 1$ with some 
number $\s \in (0, 1]$, $\s < \g$. Let $A$ be a bounded self-adjoint operator and let $P$ be an orthogonal projection 
such that $PA(\mathds{1}-P)\in \GS_{\s q}$. Then 	
\begin{equation}\label{Szego:eq}
	\|f(PAP)P - P f(A)\|_{\GS_q}
	\le C \1 f\1_n R^{\g-\s} \|PA(\mathds{1}-P)\|_{\GS_{\s q}}^\s,
\end{equation}
with a positive constant $C$ independent of the operators $A, P$, the function $f$, and the parameters $R, t_0$.  
\end{prop}

Since the operator $A$ is bounded, one does not have to 
assume that $f$ is compactly supported.  
The function $f$ can be always replaced by another function suitably localised to 
a bounded interval of size $2\|A\|$ around the origin. 
This observation allows us to obtain a bound of the correct 
degree of homogeneity. We state this fact as a corollary of Proposition \ref{Szego1:prop}.

\begin{cor} 
Suppose that the conditions of Proposition \ref{Szego1:prop} are satisfied with 
$R=\infty$. Assume in addition that $\|A\|\le 1$ and that $t_0 = 0$ 
in \eqref{fnorm:eq}. Then for any $\l >0$ we have 
\begin{equation}\label{Ab:eq}
\|f(\l PAP)P - P f(\l A)\|_{\GS_q}
\le C \1 f\1_n \l^\g \|PA(\mathds{1}-P)\|_{\GS_{\s q}}^\s, 
\end{equation} 
	with a positive constant $C$ independent of the operators $A, P$, 
	the function $f$ and the parameter $\l$.   
\end{cor}

  \begin{proof}
Let $f^{(\l)}(t) := \l^{-\g} f(\l t)$, so that 
$\1 f^{(\l)}\1_n = \1 f\1_n$.
Since $\|A\|\le 1$,   
Proposition \ref{Szego1:prop} with $R = 2$ 
leads to the bound 
\begin{equation*}
	\|f^{(\l)}(PAP)P - P f^{(\l)}(A)\|_{\GS_q}
	\le C \1 f\1_n \|PA(\mathds{1}-P)\|_{\GS_{\s q}}^\s.
\end{equation*}	 
Substituting the definition of $f^{(\l)}$ we get \eqref{Ab:eq}.
\end{proof}

As far as the $\l$-behaviour is concerned, the above estimate is sharp, since for 
$f(t) = |t|^\g$, $\g >0$, both sides have the same homogeneity in $\l$. 
We include 
such estimates where an operator (or later, a symbol) is scaled by $\lambda$ in this paper 
for completeness although the main application will appear only 
in \cite{LeSoSp-future}.

We point out one special case of the non-homogeneous function $\eta$ defined as 
\begin{equation}\label{eta:eq}
\eta(t) := -t \log |t|, t\in \R,
\end{equation}
which nevertheless leads to a homogeneous estimate:

\begin{cor}\label{Ab_eta1:cor}
	Let $q\in (0, 1]$, and let $A$ be a bounded self-adjoint operator and let $P$ be an orthogonal projection 
	such that $\|A\|\le 1$ and 
	$PA(\mathds{1}-P)\in \GS_{\s q}$ for some $\s\in (0, 1)$. Then for any $\l >0$, 	
	\begin{equation}\label{Ab_eta1:eq}
	\| \eta(\l PAP) - P \eta(\l A) P\|_{\GS_q}
	\le C_\s \l \| PA(\mathds{1}-P)\|_{\GS_{\s q}}^\s,
	\end{equation}
	with a positive constant $C_\s$ independent of the operators $A, P$ and the parameter $\l$.  
\end{cor}

\begin{proof} We write
\begin{equation*}
\eta(\l PAP) - P \eta(\l A) P = \l \bigl( \eta(PAP)P - P \eta(A)\bigr) P.
\end{equation*}
The function $\eta$ satisfies \eqref{fnorm:eq} with an arbitrary $\gamma\in (\s, 1)$, and arbitrarily 
large $n$, on any bounded interval centred at $t_0 = 0$. Now Proposition \ref{Szego1:prop} leads to the claimed estimate. 
\end{proof}

\section{Estimates for multidimensional Wiener--Hopf operators} 
\label{estim:sect}

\subsection{Definitions}
Now we derive from Proposition \ref{Szego1:prop} some estimates for 
Wiener--Hopf operators on $\plainL2(\R^d)$. 
In this paper, under Wiener--Hopf operators 
we understand operators of the form 
\eqref{WH:eq}, with a set $\L\subset\R^d$  and symbol $a = a(\bxi)$.
Throughout the paper we assume that $a\in\plainL\infty(\R^d)$ so that
the operator $\op_\a(a)$ is bounded with $\|\op_\a(a)\| = \|a\|_{\plainL\infty}$. Later we will
assume that $a$ satisfies some smoothness conditions. 
Our focus is on the operator difference
\eqref{Dalpha:eq} with suitable functions $f$. The right-hand side of 
\eqref{Dalpha:eq} is well defined for a large class of functions $f$. 
We are mostly interested in functions $f$ 
satisfying Condition \ref{f:cond}. Our immediate objective is to obtain 
for the operator \eqref{Dalpha:eq} estimates in the 
Schatten--von Neumann classes $\GS_q, q\in (0, 1]$. 
These will be derived from appropriate $\GS_q$-bounds for the operator 
\begin{equation}\label{off_diag:eq}
\chi_\L \op_\a(a) (\mathds{1}-\chi_\L).
\end{equation}
Bounds of this type were proved in \cite{Sob1}. To state them properly we need 
to specify precise conditions on the set $\L$ and the symbol $a$. 

We call a domain (an open, connected set) Lipschitz, if it can be described locally 
as a set above the graph of a Lipschitz function, 
see \cite{Sob1} for details. We call $\L$ a Lipschitz region if $\L$ is a union of 
finitely many Lipschitz domains such that their closures are pair-wise disjoint.

\begin{cond}\label{domain:cond} For $d\ge 1$ the set $\L\subset \R^d$ satisfies one of the following requirements: 
	\begin{enumerate}
		\item If $d=1$, then $\L$ is a finite union of open intervals (bounded or unbounded) such that their closures are pair-wise disjoint. 
		\item If $d\ge 2$, then $\L$ is a Lipschitz region, 
		and either $\L$ or $\R^d\setminus\L$ is bounded. 
	\end{enumerate}
\end{cond}

We rely on the bounds for the operator \eqref{off_diag:eq} obtained in \cite{Sob1}. 
They were derived for symbols $a$ satisfying the following support condition:  
\begin{equation}\label{supporta:eq}
\textup{support of the symbol $a$ is contained in 
	$B(\bmu, \tau) := \{\boldeta\in\R^d: |\boldeta-\bmu|<\tau\}$,}
\end{equation}
with some constant $\tau>0$ and $\bmu\in \R^d$. 
For methodological purposes we also introduce a smooth function $\varphi$
which is often assumed to satisfy this condition: 
\begin{equation}\label{supporth:eq}
\textup{support of the function $\varphi$ is contained in $B(\bz, \ell)$,}
\end{equation}
with some constant $\ell >0$ and $\bz\in\R^d$. 

The bounds from \cite{Sob1} also allow one to control the scaling properties through the norms 
\begin{equation}\label{norm:eq}
\SN^{(m)}(a; \tau) := \underset{0\le r\le m}
\max \ \underset{\bxi\in\R^d}
\sup \ \tau^{r}
|\nabla_{\bxi}^r a(\bxi)|, m = 1, 2, \dots,
\end{equation}
and similarlily defined norms $\SN^{(n)}(\varphi; \ell)$.
Now we can quote the result from \cite{Sob1}. The constants 
in the estimates below are independent of the symbol $a$, the function 
$\varphi$ and the parameters 
$\a, \tau, \ell$, as well as the points $\bmu$, $\bz$. 

\begin{prop}See \cite[Corollary 4.4]{Sob1} \label{cross_smooth:prop}
	Let the region $\L$ satisfy Condition 
	\ref{domain:cond},  let the symbol $a=a(\bxi)$ and 
	the function $\varphi = \varphi(\bx)$ satisfy the conditions 
	\eqref{supporta:eq} and \eqref{supporth:eq} respectively.  
	Define for some $q\in (0, 1]$ the natural numbers $m, n$ by
	\begin{equation}\label{m:eq}
	m := \lceil (d+1)q^{-1}\rceil+1, \  n := \lceil dq^{-1}\rceil+1.
	\end{equation}
	If $\a\tau\ell\ge \a_0>0$, then for any $q\in (0,1]$
	\begin{equation*}
		\|\chi_\L \varphi \op_\a(a) (\mathds{1}-\chi_\L)\|_{\GS_q}
		\le C_q (\a\tau\ell)^{\frac{d-1}{q}} 
		\SN^{(n)}(\varphi; \ell)\SN^{(m)}(a; \tau).
	\end{equation*}
	  
\end{prop} 
   
In the next subsection we extend Proposition \ref{cross_smooth:prop}  
to more general symbols $a$. 

\subsection{Multi-scale symbols, $a$} 
We consider $\plainC\infty$-symbols  
$a = a(\bxi)$ 
for which there exist positive continuous functions 
$v = v(\bxi)$ and $\tau = \tau(\bxi)$  
and constants $C_k, k=0,1,2,\ldots$ 
such that
\begin{equation}\label{scales:eq}
|a(\bxi)|\le C_0 v(\bxi), \ |\nabla_{\bxi}^k a(\bxi)|\le C_k 
\tau(\bxi)^{-k} v(\bxi),\ k = 1, 2, \dots,\quad \bxi\in\R^d.
\end{equation}
It is natural to call $\tau$ the \textit{scale (function)} 
and $v$ the \textit{amplitude (function)}. 
We  refer to symbols 
$a$ satisfying \eqref{scales:eq} as \textit{multi-scale} symbols.
In fact, in what follows, only some finite smoothness 
of the symbol $a$ is sufficient, 
but in most cases we impose the $\plainC\infty$-smoothness 
in order to avoid cumbersome formulations.
It is convenient to introduce the notation 
\begin{equation}\label{Vsigma:eq}
V_{\s, \rho}(v, \tau)  := \int \frac{v(\bxi)^\s}{\tau(\bxi)^\rho}d\bxi, \ 
\s>0, \rho\in\R.
\end{equation}
Apart from the continuity we often need some extra conditions on the scale 
and the amplitude. 
First we assume that $\tau$ is globally Lipschitz,
that is,
\begin{equation}\label{Lip:eq}
|\tau(\bxi) - \tau(\boldeta)| \le \nu |\bxi-\boldeta|,\ \ \bxi,\boldeta\in\R^d,
\end{equation}
 with some $\nu>0$.  
By adjusting the constants $C_k$ in \eqref{scales:eq} we may assume that 
$\nu<1$. It is straightforward to check that 
\begin{equation}\label{dve:eq}
(1+\nu)^{-1}\le 
\frac{\tau(\bxi)}{\tau(\boldeta)} 
\le (1-\nu)^{-1},\ \ \boldeta\in B\bigl(\bxi, \tau(\bxi)\bigr).
\end{equation}
Under this assumption on the scale 
$\tau$, the amplitude $v$ is assumed to satisfy the bounds
\begin{equation}\label{w:eq}
C_1 \le \frac{v(\boldeta)}{v(\bxi)}
\le C_2,\ \boldeta\in B\bigl(\bxi, \tau(\bxi)\bigr),
\end{equation}
with some positive constants 
$C_1, C_2$ independent of $\bxi$ and $\boldeta$.  
The condition $\nu < 1$ guarantees 
that one can construct a covering of $\R^d$ by open balls 
centred at some points $\bxi_j, j=1,2, \dots$ 
of radius $\tau_j := \tau(\bxi_j)$, which satisfies 
the \textit{finite intersection property,} that is,
the number of intersecting balls is bounded from above  
by a constant depending only on the 
parameter $\nu$, see \cite[Chapter 1, Theorem 1.4.10]{Hor}.
We denote $B_j := B(\bxi_j, \tau_j)$. Moreover, 
there exists a partition of unity $\phi_j\in\plainC\infty_0(\R^d)$ 
subordinate to the above covering such that 
\begin{equation}\label{partition:eq}
|\nabla_{\bxi}^k \phi_j(\bxi)|\le C_k \tau_j^{-k},\ k = 0, 1, \dots, \ \bxi\in\R^d,
\end{equation}
with some constants $C_k$ independent of $j = 1, 2, \dots$.

It is useful to think of $v$ and $\tau$ as 
(functional) parameters. They, in turn, can depend on other 
parameters, e.g. numerical parameters like $\a$.  
In our leading example of the Fermi symbol \eqref{positiveT:eq}, the function 
$\tau$ is naturally chosen to be dependent on the temperature $T>0$, 
see \eqref{tau_global:eq}.

\begin{rem}\label{uniform:rem}
Our aim is to derive various trace-norm estimates (resp.~asymptotics)
with explicit or implicit constants that are  independent 
of the functions $\tau$, $v$, $a$, but may 
depend on the constants in \eqref{scales:eq} and the domain $\L$. 
If the functions $\tau$, $v$ are required to satisfy  \eqref{Lip:eq} and \eqref{w:eq}, then 
the constants in the trace-norm estimates (resp.~asymptotics) may also depend on the constants $\nu$ and $C_1, C_2$ in \eqref{w:eq}. 
In all these cases we say that the estimates (resp.~asymptotics) 
are uniform in $\tau, v$ and $a$.  
%

 
In the example of the symbol 
\eqref{positiveT:eq}, the above uniformity 
allows us to control explicitly the dependence 
of the obtained bounds on the temperature. 
\end{rem}

In what follows we always assume that 
\begin{equation}\label{tauinf:eq}
\tau_{\textup{\tiny inf}} := \inf_{\bxi\in\R^d}\tau(\bxi)>0. 
\end{equation}
The constants in the obtained estimates will be independent 
of the parameters $\a, \tau_{\textup{\tiny inf}},\ell$, 
satisfying the assumption  
\begin{equation}\label{tau_low:eq}
\a \tau_{\textup{\tiny inf}}\ge \a_0, 
\end{equation}
or 
\begin{equation}\label{tau_low1:eq}
\a \ell\tau_{\textup{\tiny inf}}\ge \a_0, 
\end{equation}
with some $\a_0>0$, but may depend on $\a_0$. 

\begin{lem}\label{cross_smooth:lem} 
	Suppose that the domain $\L$ satisfies Condition 
	\ref{domain:cond}, and let 
	the functions $\tau$ and $v$ be as described above. Let $n$ 
	be as defined in \eqref{m:eq}.  
	Suppose that the symbol $a$ satisfies 
	\eqref{scales:eq}, the function $\varphi$
	satisfies \eqref{supporth:eq}, and that \eqref{tau_low1:eq} holds. 
	Then  for any $q\in (0, 1]$ we have 
	\begin{equation}\label{between_scale:eq}
		\| \chi_\L \varphi \op_\a(a) 
		(\mathds{1}-\chi_\L)\|_{\GS_q}^q
		\le C_q (\a\ell)^{d-1} 
		\bigl(\SN^{(n)}( \varphi; \ell)\bigr)^q 
		V_{q, 1}(v, \tau). 
	\end{equation}
	Suppose that \eqref{tau_low:eq} is satisfied. Then 
	\begin{equation}\label{a_multi:eq}
		\| \chi_\L \op_\a(a) (\mathds{1}-\chi_\L)\|_{\GS_q}^q
		\le C_q \a^{d-1}  V_{q, 1}(v, \tau). 
	\end{equation}
	The bound is uniform in $\tau, v$ and $a$ in the sense 
	specified in Remark \ref{uniform:rem}. 
\end{lem}

\begin{proof} 
Without loss of generality assume that $\SN^{(n)}(\varphi; \ell) = 1$.  
Let $m$ be as defined in \eqref{m:eq}. 
	Denote $v_j := v(\bxi_j)$, $\tau_j := \tau(\bxi_j)$ 
	and $B_j := B(\bxi_j, \tau_j)$, $j=1, 2, \dots$. Due to 
	\eqref{scales:eq} and \eqref{dve:eq}, \eqref{w:eq}, 
	the localised symbol $a_j  = a \phi_j$ is supported in the ball 
	$B(\bxi_j, \tau_j)$, and the bound holds:
	\begin{equation*}
		|\nabla_{\bxi}^k a_j(\bxi)|\le C_m \tau_j^{-k} v_j, \ \ k = 0, 1, 2, \dots, m,
	\end{equation*}
	so that $\SN^{(m)}(a_j; \tau_j)\le C v_j$, see \eqref{norm:eq}. 
	Since $\a\ell\tau_j\ge \a_0$, by Proposition \ref{cross_smooth:prop}, 
	we have for any $q\in (0, 1]$ that 
	\begin{equation*}
		\|\chi_\L  \varphi \op_\a(a_j)(\mathds{1}-\chi_\L)\|_{\GS_q}^q
		\le C_q (\a\ell\tau_j)^{d-1} v_j^{q},\ C_q = C_q(\a_0).
	\end{equation*}
	By the $q$-triangle inequality \eqref{qtriangle:eq} we can write
	\begin{align}\label{prom:eq}
	\| \chi_\L \varphi \op_\a(a) (\mathds{1}-\chi_\L)\|_{\GS_q}^q
	\le &\ \sum_{j}
	\| \chi_\L  \varphi
	\op_\a(a_j) (\mathds{1}-\chi_\L)\|_{\GS_q}^q\notag\\[0.2cm]
	\le &\ C_q (\a\ell)^{d-1} \sum_j  \tau_j^{d-1} v_j^{q}.
	\end{align}
	In view of \eqref{dve:eq} and \eqref{w:eq}, 
	\begin{equation*}
		\tau_j^{d-1} v_j^{q}
		\le C \underset{B_j}\int \tau(\bxi)^{-1} v(\bxi)^{q} d\bxi,
	\end{equation*}
	and hence the sum on the right-hand side of \eqref{prom:eq} is bounded by 
	\begin{equation*}
		C\sum_j  \int_{B_j}\tau(\bxi)^{-1} v(\bxi)^{q} d\bxi
		\le \tilde C \int \tau(\bxi)^{-1} v(\bxi)^{q} d\bxi.
	\end{equation*}
	At the last step we used the finite intersection property 
	of the covering $\{ B_j\}$. This leads to \eqref{between_scale:eq}.

	The bound \eqref{a_multi:eq} immediately 
	follows from \eqref{between_scale:eq} upon using a finite covering of $\L$ or $\R^d\setminus\L$ 
	by unit balls and an associated smooth partition of unity.
\end{proof}
  
Lemma \ref{cross_smooth:lem} leads to the 
following result. 

\begin{thm}\label{multi:thm} 
Suppose that $f$ satisfies Condition \ref{f:cond} 
with some $n\ge 2$ and $\g >0$, and the domain $\L$ satisfies 
Condition \ref{domain:cond}. Let $a$ be a real-valued symbol. Let 
the functions $a$ and $\tau$, $v$ be as in Lemma \ref{cross_smooth:lem}, 
and let  \eqref{tau_low:eq} be satisfied. Then for any $\s\in (0, 1],\ 
\s<\g$, and $q\in ((n-\s)^{-1}, 1]$ we have 
\begin{equation}\label{multi:eq}
\| D_\a(a, \L; f)\|_{\GS_q}^q\le C_q \a^{d-1} R^{q(\g-\s)}  \1 f\1_n^q 
\ V_{q\s, 1}(v, \tau),
\end{equation}
with a constant independent of $t_0$. 
Furthermore, if $t_0 = 0$ and $\|a\|_{\plainL\infty}\le 1$, then for any $\l >0$,
\begin{equation}\label{multi_mu:eq}
\| D_\a(\l a, \L; f)\|_{\GS_q}^q\le C_q \a^{d-1} \l^{q\g} \1 f\1_n^q 
\ V_{q\s, 1}(v, \tau).
\end{equation}
The above bounds are uniform in $\tau, v$ and $a$ in the sense 
specified in Remark \ref{uniform:rem}. Furthermore, the constants  
in \eqref{multi:eq} and \eqref{multi_mu:eq} are independent of $\a$, 
$R$, $\l$, but may depend on $\a_0$ in \eqref{tau_low:eq}.  
\end{thm}

\begin{proof}  
Use Proposition \ref{Szego1:prop} with $P = \chi_{\L}, A = \op_{\a}(a)$  
to get 
	\begin{eqnarray*}
\| D_\a(a, \L; f)\|_{\GS_q}^q&\le
&\big\|f(\chi_\L \op_{\a}(a) \chi_\L) \chi_\L - \chi_\L \op_{\a}(f\circ a)\big\|_{\GS_q}^q
\\
&\le& C_q \1 f\1_n^q R^{q(\g-\s)} \big\|\chi_\L \op_\a(a) (\mathds{1}-\chi_\L)\big\|_{\GS_{q\s}}^{q\s}. 
	\end{eqnarray*}
	To get \eqref{multi:eq} it remains to apply \eqref{a_multi:eq}. 
	The bound \eqref{multi_mu:eq} follows from 
	\eqref{a_multi:eq} and 
	\eqref{Ab:eq}. 
\end{proof}

We also state separately the estimate for the function \eqref{eta:eq}:

\begin{thm} \label{eta1:thm}
Let the function $\eta$ be as defined in \eqref{eta:eq}. Suppose 
that the real-valued symbol $a$ is as in Lemma \ref{cross_smooth:lem} 
with $\|a\|_{\plainL\infty}\le 1$ and that \eqref{tau_low:eq} is satisfied. 
Then for any $\l >0$ and any $q\in (0, 1]$, $\s\in (0, 1)$ one has 
\begin{equation}\label{multi_eta1:eq}
\| D_\a(\l a, \L; \eta)\|_{\GS_q}^q\le C_{q, \s} \a^{d-1} \l^{q}  
\ V_{q\s, 1}(v, \tau).
\end{equation}
The bound is uniform in $\tau, v$ and $a$ in the sense 
specified in Remark \ref{uniform:rem}. Furthermore, the constant 
in \eqref{multi_eta1:eq} is 
independent of $\a$, but may depend on $\a_0$ in \eqref{tau_low:eq}.  
\end{thm}
 
The proof is similar to that of \eqref{multi_mu:eq}, but instead of \eqref{Ab:eq}  
one uses \eqref{Ab_eta1:eq}. 

\section{Asymptotic results for the one-dimensional case}\label{sect:asympt}

\subsection{Results for smooth functions} 
Now we focus on the asymptotic behaviour of 
the trace of $D_\a(a, \L; f)$ as $\a\to\infty$ for dimension $d=1$. 
In line with the general theme of the paper we put the emphasis on 
non-smooth functions $f$. Our starting point, however, is the asymptotic formula 
for smooth $f$. This type of asymptotics was studied in 
\cite{Widom_82} and later in \cite{Peller}, 
and we use one result from \cite{Widom_82} without proof. 
Conditions on the smoothness and decay of the symbol $a$ 
imposed in \cite{Widom_82} are quite mild, but we assume stronger restrictions  
that enable us to utilize the bounds derived in Section \ref{estim:sect}.
More precisely, we impose the following condition. 

\begin{cond}\label{a:cond}
The symbol $a\in\plainC{m}(\R)$, $m\ge 1$ is assumed to satisfy the bound 
\eqref{scales:eq} for all derivatives up to the order $m$, 
with some continuous positive functions $\tau$ and $v$  
satisfying \eqref{Lip:eq} and \eqref{w:eq} for all $\xi\in\R$, respectively. 
\end{cond}

To state the result we first define asymptotic coefficients. For any function 
$g: \mathbb C\to\mathbb C$ and any $s_1, s_2\in\mathbb C$ denote
\begin{equation}\label{U:eq}
U(s_1, s_2; g) 
:= \int_0^1 \frac{g\bigl((1-t)s_1 + t s_2\bigr) 
	- [(1-t)g(s_1) + t g(s_2)]}{t(1-t)} dt.
\end{equation}
This integral is  
finite for functions $g\in \plainC{0, \varkappa}(\mathbb C)$,  
$\varkappa\in (0, 1]$. Note also that 
\begin{equation}\label{U_sym:eq}
U(s_1, s_1; g) = 0,\ \ \ 
U(s_1, s_2; g) = U(s_2, s_1; g), \forall s_1, s_2\in \mathbb C.
\end{equation}
Note also that the integral equals zero 
if $g(t) = 1$ or $g(t) = t$. 
Now we define the asymptotic coefficient 
\begin{equation}\label{cb:eq}
	\CB(a; g) := \frac{1}{8\pi^2}\lim_{\varepsilon\downarrow 0}
	\underset{|\xi_1-\xi_2|>\varepsilon}\iint  
	\frac{U\bigl(a(\xi_1), a(\xi_2); g\bigr)}{|\xi_1-\xi_2|^2}
	d\xi_1 d\xi_2.
\end{equation} 
Note that  $\CB$ 
is invariant under the change $a(\xi) \to a(\tau \xi)$ 
with an arbitrary $\tau >0$. If $g$ is such that $g''\in\plainL\infty(\mathbb C)$, then 
the principal value integral can be replaced by the double integral, and the following bound holds:
\begin{equation*}
|\CB(a; g)|\le C \|g''\|_{\plainL\infty} \iint 
\frac{|a(\xi_1) - a(\xi_2)|^2}{|\xi_1 - \xi_2|^2}
d\xi_1 d\xi_2.
\end{equation*}
This estimate was first pointed out in \cite[(17)]{Widom_82}. As shown in \cite{Sob_16}, 
under Condition \ref{a:cond}, one has
\begin{equation}\label{twice:eq}
|\CB(a; g)|\le C \|g''\|_{\plainL\infty}V_{2, 1}(v, \tau),
\end{equation} 
where the coefficient $V_{\s, m}$ for $\s>0,m\in\mathbb Z,$ is defined in \eqref{Vsigma:eq}.

\begin{prop}\label{Widom_82:prop} 
See \cite[Theorem 1(a)]{Widom_82}. Suppose that Condition \ref{a:cond} is satisfied with 
$m\ge 2$ and that $V_{2, 1}(v, \tau)<\infty$. 
\begin{enumerate}
\item 
Let $g$ be analytic on a neighbourhood of the closed convex hull of 
the range of the function $a$. Then the operator $D_1(a; \R_{\pm}; g)$ 
is trace class and 
\begin{equation}\label{Widom_82:eq}
\tr D_1(a; \R_{\pm}; g) = \CB(a; g).
\end{equation}
\item
If the symbol $a$ is real-valued, then formula \eqref{Widom_82:eq} holds 
under the condition $g\in \plainC{4}_0(\R)$. 
\end{enumerate}
\end{prop}

Formula \eqref{Widom_82:eq} was obtained in \cite{Widom_82} 
under weaker conditions on the symbol $a$. Moreover, for real-valued 
symbols $a$ the smoothness conditions on $g$ 
in \cite{Widom_82} are less restrictive than in the above proposition.
Note also that for real-valued $a$ the paper \cite{Peller} 
allows further relaxation on the functions $a$ and $g$ but we omit the details.

\medskip
By rescaling $ x\to \a x$  
one immediately concludes that the left-hand side 
of \eqref{Widom_82:eq} coincides with $\tr D_\a(a; \R_{\pm}; g)$. 
It is worth pointing out that, formally speaking, the estimates in Section 
\ref{estim:sect} do not ensure that the 
trace on the left-hand side of \eqref{Widom_82:eq} 
is finite, since neither $\R_\pm$ itself nor its complement is bounded. 
However, those estimates in combination with Proposition \ref{separate:prop} below do guarantee 
that $D_1(a; \R_{\pm}; g)$ is trace-class.

We apply Proposition \ref{Widom_82:prop} to the case of a real-valued 
symbol $a$ and the function $g:\R\mapsto \mathbb C$ defined as
\begin{equation*}
g(\l) := r_z(\l) := \frac{1}{\l-z},\  \im z\not = 0.
\end{equation*}  
Now our immediate objective is to derive from 
\eqref{Widom_82:eq} a similar asymptotic 
formula for the operator $D_\a(a; \L; g)$ with 
a set $\L$ satisfying Condition \ref{domain:cond}(1).  
For $d=1$, instead of $\L$ we use the notation $I$. According  
to Condition \ref{domain:cond}(1),
\begin{equation}\label{Kint:eq}
I = I_0\cup I_{K+1}\underset{k=1}{\overset{K}\bigcup} I_k 
\end{equation}
where $\{I_k\}, k = 1, 2, \dots, K$ is a finite 
collection of bounded open intervals such that 
their closures are disjoint, the set $I_0$ (resp.~$I_{K+1}$) 
is either empty or $(-\infty, x_0)$ (resp.~$(x_0, \infty)$) 
with some $x_0\in\R$, and its closure is also disjoint from the 
other intervals. Below we use the following notation
for the number of endpoints of $I$, namely
\begin{equation}\label{omega:eq}
\omega := 
\begin{cases}
2K\ \textup{if } I_0 = I_{K+1} = \varnothing,\\
2K+1\ \textup{if only one of $I_0, I_{K+1}$ is non-empty},\\
2K+2 \ \textup{if both $I_0, I_{K+1}$ are non-empty}. 
\end{cases}
\end{equation}
By writing $K=K(I)$ and $\om=\om(I)$ we emphasize the dependence on 
the set $I$. We observe that 
\begin{equation}\label{compl:eq}
\om(I) = \om(I^c),\ \textup{with}\ I^c = \R\setminus I.
\end{equation}

For arbitrary symbols $a, b$ we introduce the notation 
\begin{equation}\label{M:eq}
M^{(m)}(a, b) := \|\p_\xi^m a\|_{\plainL1} \|b\|_{\plainL\infty} 
+ \|a\|_{\plainL\infty}\|\p_\xi^m b\|_{\plainL1},\ m= 1, 2, \dots,
\end{equation}
and denote 
\begin{equation}\label{distz:eq}
\d(z, a) := \dist(z, [-\|a\|_{\plainL\infty}, \|a\|_{\plainL\infty}]) >0.
\end{equation}

\begin{thm}\label{interval_res:thm}
Let $I$ and $\omega$ be as described in \eqref{Kint:eq} and \eqref{omega:eq}. Assume that 
\begin{equation}\label{mult_int:eq}
\inf_{k, j: k\not = j} 
\{|I_k|, \dist(I_k, I_j)\}\ge  1, \ \ k, j = 0, 1, 2, \dots, K+1.
\end{equation}
Suppose that $a\in\plainC{m}(\R), m\ge 3,$ is real-valued. Then for any $\a>0$ we have  
\begin{align}\label{int_res:eq}
\bigl|\tr D_\a(a, I; r_z) - &\ \omega\CB(a; r_z)\bigr|\notag\\[0.2cm]
\le  &\ C_m \a^{-m+1} \frac{1}{\d(z, a)}
\biggl(\frac{|z|+\|a\|_{\plainL\infty}}{\d(z, a)}\biggr)^3 M^{(m)}(a_z, a_z^{-1}),
\end{align}
with a constant $C_m$ independent of $a, z, \d(z, a)$, and $\a$, and the intervals 
$I_k$, $k = 0, 1, \dots$, $K+1$.
\end{thm}

Clearly, by scaling we may replace the $1$ on the right-hand side in condition 
\eqref{mult_int:eq} by any (strictly) positive real number.

In the case of one bounded interval, the convergence of 
the left-hand side of \eqref{int_res:eq} 
to zero as $\a\to\infty$ was proved 
in \cite[Theorem 2]{Widom_82}, 
see also \cite[Theorem 9]{Peller}. 
Note that for an infinitely smooth $a$ the right-hand 
side of \eqref{int_res:eq} decays as 
$\a^{-\infty}, \a\to\infty$. 
For one bounded interval, this effect was pointed out in \cite[formula (1.5)]{BuBu}. 
These conclusions of \cite{BuBu, Peller, Widom_82} 
are not sufficient for us, as our 
aim is to have a more explicit control of 
the remainder as a function of the symbol $a$ as in Theorem \ref{interval_res:thm}. 
In particular, when considering symbols $a = a_{T,\mu}$ defined in 
\eqref{positiveT:eq}, Theorem \ref{interval_res:thm} allows us to obtain 
estimates that depend explicitly on the temperature $T$, and possibly the chemical potential $\mu$.  
The proof of Theorem \ref{interval_res:thm} draws on the ideas of \cite{Widom_82}, and 
it is postponed until Section \ref{res1:sect}.

We extend the above bound to arbitrary functions of finite smoothness satisfying 
some decay conditions. Precisely, for $g\in \plainC{n}(\R)$, $n\in\mathbb N_0$ and a  
constant $r >0$ we define 
\begin{equation}\label{Nn:eq}
	N_{n}(g)  := N_n(g; r) := \sum_{k=0}^n\int |g^{(k)}(t)| \lu t\ru_r^{k-2} dt,\ 
	\lu t\ru_r := \sqrt{t^2 + r^2}.
	\end{equation}
 
\begin{thm}\label{interval_fn:thm}
Let $I$ and $\omega$ be as in the previous theorem. Suppose that $a\in \plainC{m}(\R), m\ge 3,$ 
is real-valued and satisfies the bound \eqref{scales:eq} with some continuous positive functions $\tau, v$.  Suppose further that $f\in \plainC{n}_0(\R)$ with $n\ge m+6$. Then 
for any $r\ge \|v\|_{\plainL\infty}$ and any $\a>0$ we have 
\begin{equation}\label{fn:eq}
	\bigl|\tr D_\a(a, I; f) - \omega \CB(a; f)\bigr|
	\le  C_{m, n}  N_n(f; r)  
	\a^{-m+1}	V_{1, m}(v, \tau).
\end{equation}
This bound is uniform in $\tau, v$ and $a$ in the sense specified 
in Remark \ref{uniform:rem}. The constant $C_{m,n}$ in \eqref{fn:eq} is independent 
of the parameters $\a, r$ and the function $f$.
\end{thm}

\subsection{Results for non-smooth functions}
Now we assume that $f$ satisfies Condition \ref{f:cond} with some $\g >0$. 
In this case, if $\gamma>0$ is small,  
it is not immediately clear why and under which conditions on the symbol $a$ 
the coefficient $\CB(a; f)$ is finite. This issue was investigated in 
\cite{Sob_16}. We quote the appropriate bound, adjusted for the use in the forthcoming 
calculations. We use the integral $V_{\s, \rho}(v, \tau)$ defined in \eqref{Vsigma:eq} and
the notation $\varkappa :=  \min\{1, \g\}$.

\begin{prop}\label{scales:prop} See \cite[Theorem 6.1]{Sob_16}.
Suppose that $f$ satisfies Condition \ref{f:cond} with $n = 2$, 
$\g>0$ and some $R>0$. 
Let $a\in\plainC\infty(\R)$ satisfy Condition \ref{a:cond}. 
Then for any $\s\in (0, \varkappa]$ we have 
\begin{align}\label{coeffscales:eq}
|\CB(a; f)|\le C_\s \1 f\1_2 R^{\g-\s} V_{\s, 1}(v, \tau),
\end{align}
with a constant $C_\s$ independent of $f$, 
uniformly in the functions $\tau, v$, 
and the symbol $a$ in the sense specified in Remark 
\ref{uniform:rem}. 
\end{prop}

We note another useful result from \cite{Sob_16}. It describes the contribution of ``close" points 
$\xi_1$ and $\xi_2$ in the coefficient \eqref{cb:eq}. Suppose that 
$\tau_{\textup{\tiny inf}} := \inf_{\xi\in\R} \tau(\xi)>0$, then we define 
\begin{equation}\label{cb1:eq}
\CB^{(1)}(a; f) := \frac{1}{8\pi^2}\lim_{\varepsilon\downarrow 0}
	\underset{\varepsilon<|\xi_1-\xi_2| < \frac{\tau_{\textup{\tiny inf}}}{2}}\iint  
	\frac{U\bigl(a(\xi_1), a(\xi_2); f\bigr)}{|\xi_1-\xi_2|^2}
	d\xi_1 d\xi_2.
\end{equation}
This quantity is estimated in the following proposition. 
 
\begin{prop}\label{sub:prop} 
Suppose that $f$ satisfies Condition \ref{f:cond} with $n = 2$ and $\g>0$. 
Let $a\in\plainC\infty(\R)$ satisfy Condition \ref{a:cond}.
Suppose also that $\tau_{\textup{\tiny inf}} > 0$.
Then for any $\d\in [0, \varkappa)$,
the following bound holds:
\begin{equation}\label{sub:eq}
|\CB^{(1)}(a; f)|\le 
	C_\d\1 f\1_2 \tau_{\textup{\tiny inf}}^\d V_{\varkappa, 1+ \d}(v, \tau),
\end{equation}
uniformly in the functions $\tau, v$, 
and the symbol $a$ in the sense specified in Remark 
\ref{uniform:rem}.
 \end{prop} 
This bound follows from \cite[Corollary 6.5]{Sob_16}.  
 
\bigskip
The bound \eqref{coeffscales:eq} plays a central role in the proof of the 
following theorems. From now on we assume 
that $\tau_{\textup{\tiny inf}}  >0$ 
and that $\tau_{\textup{\tiny inf}}$ and $\a$ satisfy \eqref{tau_low:eq}. 
The convergence in the next theorems is uniform in 
the functions $\tau, v$, and the symbol $a$ in the sense specified in Remark 
\ref{uniform:rem}, but no 
uniformity is claimed in the parameter $\a_0$ in \eqref{tau_low:eq}.
 
\begin{thm} \label{scale_asymp:thm} 
Let $I$ and $\omega$ be as described in \eqref{Kint:eq} and \eqref{omega:eq}. 
Suppose that $f$ satisfies Condition \ref{f:cond} with some $\g >0$, $n = 2$ and 
some $t_0\in\R$. Let $a\in\plainC\infty(\R)$ 
be a real-valued symbol satisfying Condition 
	\ref{a:cond}, and let $\a\tau_{\textup{\tiny inf}}\ge \a_0$.   
	Suppose that  
	$\|v\|_{\plainL\infty}\le 1$ 
and $V_{\s, 1}(v, \tau)<\infty$ for some $\s\in (0, 1]$, $\s < \g$, and 
\begin{equation}\label{higher:eq}
\underset{\a\to\infty}\lim\a^{-m+1} 
\frac{V_{1, m}(v, \tau)}{V_{\s, 1}(v, \tau)} = 0,
\end{equation}
uniformly in $v, \tau$ (see Remark \ref{uniform:rem}), for some $m \ge 3$. Then  
\begin{equation}\label{scale_asymp:eq}
	\underset{\a\to\infty}
	\lim \frac{1}{V_{\s, 1}(v, \tau)}
	\bigl(\tr D_\a(a,I; f) - \omega\CB(a; f)\bigr) = 0,
\end{equation}
and the convergence is uniform in $v, \tau$ and $a$.
\end{thm}
  
In order to avoid possible confusion we recall that 
$v$, $\tau$ are thought of as functional parameters of the problem,  
and they may depend on the numerical parameter    
 $\a$. Thus the equality \eqref{higher:eq} is a genuine, 
 non-vacuous assumption.
 
For the next theorem recall that the function $\eta$ is defined in \eqref{eta:eq}.

\begin{thm} \label{smalll_asymp:thm} 
Let $I$ and $\omega$ be as in the previous theorem.
	Suppose that $f$ satisfies Condition \ref{f:cond} with some $\g >0$, $t_0 = 0$,
	and all $n$. Let $a\in\plainC\infty(\R)$ be a real-valued 
	symbol satisfying 
	Condition \ref{a:cond}, and let $\a\tau_{\textup{\tiny inf}}\ge \a_0$.     
	Suppose that $\|v\|_{\plainL\infty}\le 1$ 
	and $V_{\s, 1}(v, \tau)<\infty$ for some $\s\in (0, 1]$, $\s < \g$, and that 
	\eqref{higher:eq} is satisfied. 
	 Then  
	for any real $\l>0$, one has
	\begin{equation}\label{smalll_asymp:eq}
	\underset{\a\to\infty}
	\lim \frac{1}{\l^\g V_{\s, 1}(v, \tau)}
	\bigl(\tr D_\a(\l a, I; f) 
	- \omega\CB(\l a; f)\bigr)	= 0. 
	\end{equation}
In addition, if $V_{\s, 1}(v, \tau) <\infty$ with some $\s < 1$, then 
	\begin{equation}\label{smalll_eta1:eq}
	\underset{\a\to\infty}\lim 
	\frac{1}{\l V_{\s, 1}(v, \tau)}\bigl(\tr D_\a(\l a, I; \eta) - 
	\omega \CB( \l a; \eta)\bigr)= 0. 
	\end{equation}
The convergence in 
\eqref{smalll_asymp:eq} and \eqref{smalll_eta1:eq} 
is uniform in 
the functions $\tau, v$, and the symbol $a$ in the sense specified in Remark 
\ref{uniform:rem}, and in 
\eqref{smalll_eta1:eq} it is also uniform in 
the parameter $\l\in (0, \l_0]$ for any $\l_0<\infty$. 	
\end{thm}
 
 We point out that the smoothness 
conditions on the 
function $f$ in Theorem \ref{smalll_asymp:thm} 
are much more restrictive than those in Theorem \ref{scale_asymp:thm}. 
This difference
will be briefly explained after the proof of Theorem \ref{smalll_asymp:thm}.
 
The main difficulty lies in the proof of Theorem \ref{interval_res:thm}, whereas the 
remaining theorems are derived from it via relatively standard methods. In the next 
section we concentrate on this derivation. The proof of Theorem \ref{interval_res:thm} 
is deferred until Section \ref{res1:sect}. 

\section{Proofs of Theorems \ref{interval_fn:thm}, \ref{scale_asymp:thm} \& \ref{smalll_asymp:thm}}\label{proofs:sect}

We use the almost analytic extension constructed in Lemma \ref{ab:lem} 
with some $r \ge \|v\|_{\plainL\infty}$, where $v$ is the amplitude of the symbol 
$a$ as in \eqref{scales:eq}. Let $\tilde f$ be the almost analytic extension of $f$ 
constructed in Lemma \ref{ab:lem}.  
It follows from \eqref{HS:eq} that   
\begin{align*}
	\tr D_\a(a, I; f) 
	- &\ \omega\CB(a; f)\\[0.2cm]
	= &\ \frac{1}{\pi}\iint \frac{\p}{\p \overline{z}}\tilde f(x, y; r) \bigl(\tr D_\a(a,I; r_z) 
	- \omega\CB(a; r_z)\bigr) dx dy.
\end{align*}
Thus by Theorem \ref{interval_res:thm} and by \eqref{ab:eq} we have
\begin{align}\label{trans:eq}
	|\tr D_\a(a, I; f) 
	- &\ \omega\CB(a; f)|\notag\\[0.2cm]
	\le &\ C \a^{-m+1}  
	\iint \Big|\frac{\p}{\p \overline{z}}\tilde f(x, y; r)\Big| \,
	\bigl(|x|+|y| + \|a\|_{\plainL\infty}\bigr) ^3|y|^{-4} M^{(m)}(a_z, a_z^{-1}) dx dy\notag\\[0.2cm]
	\le &\ C\a^{-m+1} 
	\int \underset{|y|< \lu x\ru_r}\int
	F(x; r)\lu x\ru_r^3  |y|^{n-5} M^{(m)}(a_z, a_z^{-1}) dy dx, 
\end{align} 
for any $r\ge \|v\|_{\plainL\infty}$. 
Let us now estimate $M^{(m)}(a_z, a_z^{-1})$. 

\begin{lem}\label{az:lem} 
	Suppose that $a\in\plainC{m}(\R)$ satisfies 
	\eqref{scales:eq} with some $m\ge 1$. 
	%
	%
	Then 
	\begin{equation}\label{M_above:eq}
	M^{(m)}(a_z, a_z^{-1})\le C_m 
	\frac{\|v\|_{\plainL\infty}+|z|}{|\im z|^2}
	\biggl( 1 + \frac{\|v\|_{\plainL\infty}^{m-1}}{|\im z|^{m-1}}	\biggr)
	V_{1, m}(v, \tau). 
	\end{equation}
	Moreover, for any $r\ge \| v\|_{\plainL\infty}$, and all 
	$y$ with  $|y|< \lu x\ru_r$, we have 
	\begin{equation}\label{M_above1:eq}
	M^{(m)}(a_z, a_z^{-1})\le C_m \frac{\lu x\ru_r }{|y|^2}\biggl(
	1+ \frac{\lu x\ru_r^{m-1}}{|y|^{m-1}}
	\biggr)V_{1, m}(v, \tau),
	\end{equation}
	with a constant $C_m$ independent of $r$. 
	\end{lem}

\begin{proof} 
	By definition \eqref{M:eq},
	\begin{equation}\label{M1:eq}
	M^{(m)}(a_z, a_z^{-1}) = \| \p_{\xi}^m a\|_{\plainL1} \|a_z^{-1}\|_{\plainL\infty} 
	+ \| \p_{\xi}^m a_z^{-1}\|_{\plainL1} \|a_z\|_{\plainL\infty} .
	\end{equation}
	In view of the bound \eqref{scales:eq} 
	the first summand in the above formula 
	is bounded by 
	\begin{equation*}
		C_m\frac{1}{|\im z|}\int 
		\tau(\xi)^{-m} v(\xi) d\xi.  
	\end{equation*}
	To estimate the second term on the right-hand side of \eqref{M1:eq} we use
	the Leibniz formula and \eqref{scales:eq} to obtain 
	\begin{equation*}
		|\p_{\xi}^m a_z^{-1}|\le C_m 
		\biggl(\frac{v}{|\im z|^2}+	
		\frac{v^m}{|\im z|^{m+1}}	\biggr)
		\tau^{-m}.
	\end{equation*}
	Thus the second summand in \eqref{M1:eq} does not exceed
	\begin{equation*}
		C_m (\|v\|_{\plainL\infty} + |z|) \biggl(\frac{1}{|\im z|^2}+	
		\frac{\|v\|_{\plainL\infty}^{m-1}}{|\im z|^{m+1}}\biggr)\int 
		\tau(\xi)^{-m} v(\xi) 
		d\xi. 
	\end{equation*}
	This leads to the claimed bound \eqref{M_above:eq}. 
	
	For $r\ge \|v\|_{\plainL\infty}$ and $|y|< \lu x\ru_r$, 
	the bound \eqref{M_above1:eq} 
	immediately follows from \eqref{M_above:eq}.
\end{proof}

\begin{proof}[Proof of Theorem \ref{interval_fn:thm}]
By \eqref{M_above1:eq}, the integral on the 
	right-hand side of \eqref{trans:eq} is estimated by
	\begin{equation*}
		V_{1, m}(v, \tau)	\int \underset{|y|< \lu x\ru_r}\int
			F(x; r)\lu x\ru_r^4  |y|^{n-7} 
			\biggl(
			1+ \frac{\lu x\ru_r^{m-1}}{|y|^{m-1}}
			\biggr)dy dx.
		\end{equation*}
	Since $n \ge m+6$, this integral is finite and it is bounded by 
	\begin{equation*}
		C V_{1, m}(v, \tau) \int F(x;r) \lu x \ru_r^{n-2} dx 
		= C V_{1, m}(v, \tau) N_n(f; r),
	\end{equation*}
	where we have used the definition \eqref{Nn:eq}.
	Therefore \eqref{trans:eq} yields the bound
	\begin{equation*}
		|\tr D_\a(a, I; f) - \omega\CB(a; f)| 
		\le	C N_n(f; r) \a^{-m+1} V_{1, m}(v, \tau),
	\end{equation*}
	as claimed. 
\end{proof}

\begin{proof}[Proof of Theorem \ref{scale_asymp:thm}] 
For brevity we denote $D_\a(f) := D_\a(a,I; f)$ and $\CB(f) := \CB(a; f)$. 

\underline{Step 1: proof of formula \eqref{scale_asymp:eq}  
for $f\in \plainC2(\R)$.} 
Without loss of generality we may assume that  
$\|a\|_{\plainL\infty}\le 1/2$ and that 
the function $f$ is supported on the interval $[-1, 1]$. 
By the Weierstrass Theorem, for any $\varepsilon>~0$ 
one can find a real polynomial 
$f_\varepsilon$ such that the function 
$g_\varepsilon := f-f_\varepsilon$ satisfies the bound 
\begin{equation}\label{feps:eq}
\max_{0\le k\le 2}\max_{ |t|\le 1}|g_\varepsilon^{(k)}(t)|< \varepsilon.
\end{equation}
Clearly,
\begin{equation*}
D_\a(f) = D_\a(f_\varepsilon) + D_\a(g_\varepsilon).
\end{equation*}
In order to estimate $D_\a(g_\varepsilon)$ we extend the function $g_\varepsilon$ 
to the interval $[-2, 2]$ as a $\plainC2_0$-function in such a way that 
$\|g_\varepsilon\|_{\plainC2} \le C\varepsilon$ with some universal 
constant $C>0$. 
Observe now that 
such $g_\varepsilon$ satisfies Condition \ref{f:cond} with $t_0 = -3$, $R=5$, $n=2$  
and arbitrary $\g >0$. Furthermore, $\1 g_\varepsilon\1_2 < C\varepsilon$. 
To be definite we take $\gamma=2$. 
Since the condition \eqref{tau_low:eq} is satisfied, 
we may use Theorem \ref{multi:thm} 
with $q = 1$ and arbitrary $\s \in (0, 1)$, so that  
\begin{equation}\label{geps:eq}
\|D_\a(g_\varepsilon)\|_{\GS_1}\le C \varepsilon V_{\s, 1}(v, \tau).
\end{equation}
Moreover, by \eqref{twice:eq},
\begin{equation}\label{bgeps:eq}
|\CB(g_\varepsilon)|\le C\varepsilon V_{2, 1}(v, \tau)
\le C\varepsilon V_{\s, 1}(v, \tau).
\end{equation}
In order to handle the trace of $D_\a(f_\varepsilon)$, extend 
the polynomial 
$f_\varepsilon$ as a $\plainC\infty_0$-function on the interval $[-2, 2]$. Thus by Theorem 
\ref{interval_fn:thm} with $r = 1\ge \|v\|_{\plainL\infty}$,
\begin{align*}
|\tr D_\a(f_\varepsilon) - \omega \CB(f_\varepsilon)|\le 
C N_n(f_\varepsilon; 1)
\a^{-m+1}V_{1, m}(v, \tau),
\end{align*}
with arbitrary $m\ge 3$. 
In view of the condition \eqref{higher:eq} 
and by virtue of \eqref{geps:eq} and \eqref{bgeps:eq}, 
we have 
\begin{align*}
\limsup&\,\frac{1}{V_{\s, 1}(v, \tau)}\big|\tr D_\a(f) 
- \ \omega\CB(f)\big|\\[0.2cm] 
\le &\ \limsup\frac{1}{V_{\s, 1}(v, \tau)}\big|\tr D_\a(f_\varepsilon) 
- \omega\CB(f_\varepsilon)\big|\\[0.2cm]
+ &\ \limsup\frac{1}{V_{\s, 1}(v, \tau)}
\|D_\a(g_\varepsilon)\|_{\GS_1} 
+ \omega\limsup\frac{1}{V_{\s, 1}(v, \tau)}
|\CB(g_\varepsilon)|\\[0.2cm]
\le &\ 
 N_n(f_{\varepsilon}; 1) \ \limsup\ \a^{-m+1} 
\frac{V_{1, m}(v, \tau)}{V_{\s, 1}(v, \tau)}
+ C\varepsilon =  C\varepsilon.
\end{align*}
Here the $\limsup$ is taken as $\a\to\infty$, 
$\a\tau_{\textup{\tiny inf}}\ge \a_0$, and 
it is uniform in $v, \tau$ and $a$. 
Since $\varepsilon>0$ is arbitrary, this leads to 
\eqref{scale_asymp:eq} for arbitrary $\plainC2$-functions $f$.

\underline{Step 2. Completion of the proof.} 
Let $f$ be a function as specified in the theorem.  
Let $\z\in\plainC\infty_0(\R)$ be a real-valued function satisfying \eqref{zeta:eq}.
We represent 
\begin{align*}
f= &\ f_R^{(1)}+ f_R^{(2)}, 0<R\le 1,\\[0.2cm]
f_R^{(1)}(t) := &\ f(t) \z\bigl((t-t_0)R^{-1}\bigr),\\[0.2cm] 
f_R^{(2)}(t) := &\ f(t)  - f_R^{(1)}(t).
\end{align*}
For $f_R^{(1)}$ we use Theorem \ref{multi:thm} 
with $q = 1$, $n=2$, and a $\s\in (0, 1], \s <\g,$ such that $V_{\s, 1}<\infty$:
\begin{equation*}
\|D_\a(f_R^{(1)})\|_{\GS_1}\le C R^{\g-\s} \1 f_R^{(1)}\1_2 V_{\s, 1}(v, \tau). 
\end{equation*}
By \eqref{coeffscales:eq}, the coefficient $\CB(f_R^{(1)})$ 
satisfies the same bound. Note also that $\1 f_R^{(1)}\1_2\le C\1 f\1_2$, 
so that
\begin{align}\label{g1R:eq}
\frac{1}{V_{\s, 1}(v, \tau)}
\bigl|D_\a(f_R^{(1)})- \omega\CB({f_R^{(1)}})\bigr|
\le C \1 f\1_2 R^{\g-\s}.
\end{align}
Now, it is clear that  $f_R^{(2)}\in \plainC2(\R)$, 
so one can use formula \eqref{scale_asymp:eq} established in  
Part 1 of the proof:
\begin{equation*}
\underset{\a\to\infty, \ \a\tau_{\textup{\tiny inf}}\ge \a_0}
\lim \ \frac{1}{V_{\s, 1}(v, \tau)}
|\tr D_\a(f_R^{(2)}) - 
\omega\CB(f_R^{(2)})|  = 0. 
\end{equation*}
Together with \eqref{g1R:eq}, this equality gives
\begin{align*}
\limsup \ \frac{1}{V_{\s, 1}(v, \tau)}
&\ |\tr D_\a(f) - \omega\CB(f)|\\[0.2cm]
\le &\ 
\lim \ \frac{1}{V_{\s, 1}(v, \tau)}
|\tr D_\a(f_R^{(2)}) - \omega\CB(f_R^{(2)})|  + C \1 f\1_2 R^{\g-\s}\\[0.2cm]
\le &\  C\1 f\1_2 R^{\g-\s}. 
\end{align*}
Again, the limits above are taken as $\a\to\infty$, $\a\tau_{\textup{\tiny inf}}\ge \a_0$. 
Since $\s<\g$, and $R>0$ is arbitrary, the required asymptotics follow. 
\end{proof}

\begin{proof} [Proof of Theorem \ref{smalll_asymp:thm}]
 	 	Instead of $f$ we introduce for $\l>0$ the function 
 	\begin{equation*}
 	f^{(\l)}(t) := \l^{-\g} f(\l t), t\in \R.
 	\end{equation*}
 		It is clear that $\1 f\1_n = \1 f^{(\l)}\1_n$ for all $n$. 
 	As in the previous proof, without loss of generality we may assume that 
        $\|a\|_{\plainL\infty}\le 1/2$, so that the function $f^{(\l)}$ may be assumed 
        to be supported on the interval $[-1, 1]$. Note that 
 	\begin{equation*}
 	D_\a (\l a, I;f) = \l^\g D_\a(a,I; f^{(\l)}),
 	\ \ \CB(\l a; f) = \l^\g\CB(a; f^{(\l)}).
 	\end{equation*}
Thus the asymptotic formula \eqref{smalll_asymp:eq} is equivalent to the following relation:
	\begin{equation}\label{smalll_asymp1:eq}
	\underset{\a\to\infty, \a\tau_{\textup{\tiny inf}}\ge \a_0}
	\lim \frac{1}{ V_{\s, 1}(v, \tau)}
	\bigl(\tr D_\a(a, I; f^{(\l)}) - \omega\CB(a; f^{(\l)})\bigr)
	= 0. 
	\end{equation}
The further proof now follows essentially Step 2 of the proof of Theorem \ref{scale_asymp:thm}. 
As before, for brevity we use the notation $D_\a(f) := D_\a(a, I; f)$, $\CB(f) := \CB(a; f)$. 
 	
Let $\z\in\plainC\infty_0(\R)$ be a real-valued function, satisfying \eqref{zeta:eq}.
 	Represent 
 	\begin{align*}
 	f^{(\l)}= &\ g_R^{(1)}+ g_R^{(2)}, 0<R\le 1,\\[0.2cm]
 	g_R^{(1)}(t) := &\ f^{(\l)}(t) \z\bigl(tR^{-1}\bigr),\\[0.2cm] 
 	g_R^{(2)}(t) := &\ f^{(\l)}(t)  - g_R^{(1)}(t).
 	\end{align*}
Since the condition \eqref{tau_low:eq} is satisfied, for $g_R^{(1)}$ we may use Theorem \ref{multi:thm} 
 	with $q = 1$, $n=2$, and a $\s\in (0, 1], \s <\g,$ such that $V_{\s, 1}<\infty$:
 	\begin{equation*}
 	\|D_\a(g_R^{(1)})\|_{\GS_1}\le C R^{\g-\s} \1 g_R^{(1)}\1_2 V_{\s, 1}(v, \tau)
 	\le C R^{\g-\s} \1 f \1_2 V_{\s, 1}(v, \tau). 
 	\end{equation*}
  By \eqref{coeffscales:eq} the coefficient 
 	$\CB(g_R^{(1)})$ satisfies the same bound. 
 	It is clear that  $g_R^{(2)}\in \plainC\infty(\R)$, 
 	and by definition \eqref{Nn:eq},  
 	\begin{equation*}
 	N_n(g_R^{(2)}; r)\le C_{n, R, r} \1 f^{(\l)}\1_n
 	\le \tilde C_{n, R, r}\1 f\1_n, n = 1, 2, \dots,
 	\end{equation*}
 	for any $r >0$. 
 	Thus by \eqref{fn:eq},
 	\begin{align*}
 	|\tr D_\a(g_R^{(2)}) - \omega \CB(g_R^{(2)})|\le C N_n(g_R^{(2)}; 1+\l_0)
 	\a^{-m+1} V_{1, m}(v, \tau),
 	\end{align*}
 	with arbitrary $m\ge 3$.
Therefore, using \eqref{higher:eq} and arguing as in the proof of Theorem 
\ref{scale_asymp:thm}, we obtain
 	\begin{align*}
 	\limsup &\ \frac{1}{V_{\s, 1}(v, \tau)}
 	\ |\tr D_\a(f^{(\l)}) - \omega\CB(f^{(\l)})|\\[0.2cm]
 	\le &\ 
 	\lim \ \frac{1}{V_\s(v, \tau)}
 	|\tr D_\a(g_R^{(2)}) - \omega
 	\CB(g_R^{(2)})|  + C \1 f\1_2 R^{\g-\s}\\[0.2cm]
 	\le &\  C\1 f\1_2 R^{\g-\s}. 
 	\end{align*}
 	The limits above are taken as $\a\to\infty$, $\a\tau_{\textup{\tiny inf}}\ge \a_0$. 
 	Since $R>0$ is arbitrary, the required asymptotics \eqref{smalll_asymp1:eq} follow. 
 	As explained previously, this implies \eqref{smalll_asymp:eq}. 
 	
 	Proof of \eqref{smalll_eta1:eq}. We write 
 	\begin{equation*}
 	D_\a(\l a, I; \eta) = \l D_\a(a, I; \eta),\ 
 	\CB(\l a; \eta) = \l \CB(a; \eta).
 	\end{equation*}
 	Since the function $\eta$ satisfies Condition \ref{f:cond} for any $\g\in (\s, 1)$, 
 	the proclaimed asymptotic formula is a direct consequence of the formula 
 	\eqref{scale_asymp:eq} for the operator $D_\a(a, I; \eta)$. 
\end{proof}

Observe that the proof of Theorem \ref{smalll_asymp:thm} has only one step, 
in contrast to that of Theorem \ref{scale_asymp:thm}. Namely, 
in the former we do not prove that the sought asymptotics holds for arbitrary 
$f\in\plainC2(\R)$ since this would require approximating $f^{(\l)}$ 
with polynomials whose dependence on $\l$ would have to be explicitly 
controlled. We do not go into these difficulties.

\section{Proof of Theorem \ref{interval_res:thm}: the case of a single interval}\label{res1:sect}

We recall the notation \eqref{WH:eq} for 
the Wiener--Hopf operator: 
$W_\a(a;I) = \chi_I \op_\a(a) \,\chi_I$ with the 
notation $\L$ replaced by a subset $I\subset\mathbb R$. 
A central role in our argument plays the operator 
\begin{equation}\label{halpha:eq}
 H_\a(a, b; I) := W_\a(a b; I) - W_\a(a; I) W_\a(b; I) 
 = \chi_I \op_\a(a) \bigl(\mathds{1} - \chi_I\bigr) \op_\a(b) \chi_I.
\end{equation} 
with $\plainC{m}$-symbols $a = a(\xi)$ and $b = b(\xi)$. 
At the first step of the proof we 
assume that the set $I$ is just a bounded 
interval $(x_0, y_0)$ with $y_0-x_0\ge 1$. 

\subsection{Preliminary bounds} 
For any $z\in\R$  denote 
\begin{equation*}
\R_{z}^{(+)} := (z, \infty), \R_{z}^{(-)} := (-\infty, z), \ \ 
\chi_{z}^{(+)} := \chi_{(z, \infty)}, \chi_{z}^{(-)} := \chi_{(-\infty, z)}.
\end{equation*} 
We define 
\begin{equation*}
Z_\a(a, b; I)
:= H_\a(a, b; I) - H_\a(a, b; \R_{y_0}^{(-)}) - H_\a(a, b; \R_{x_0}^{(+)}). 
\end{equation*}
Most of the estimates for the introduced operators will follow from 
the next proposition, which is a consequence of \cite[Theorem 2.7]{Sob1}.

\begin{prop}\label{separate:prop}
Let $a$ be a symbol such that $\p^m_\xi a\in\plainL1(\R)$ with some $m\ge 3$. 
Let $z, t$ be numbers such that $z-t = \ell>0$. 
Then for any $\a > 0$,  we have
\begin{equation*}
\|\chi_t^{(-)} \op_\a(a) \chi_z^{(+)}\|_{\GS_1}
+ \|\chi_z^{(+)} \op_\a(a) \chi_t^{(-)}\|_{\GS_1}
\le C_m (\a\ell) ^{-m+1} \|\p_\xi^m a\|_{\plainL1}.
\end{equation*}
\end{prop}

\begin{proof} 
The operator $\chi_t^{(-)}\op_\a(a) \chi_z^{(+)}$ is trivially unitarily equivalent 
to $\chi_0^{(-)}\op_1(b) \chi_1^{(+)}$ with the symbol $b$ defined as
$b(\xi) := a\bigl((\a\ell)^{-1}\xi\bigr)$. By \cite[Theorem 2.7]{Sob1}, 
\begin{equation*}
\|\chi_0^{(-)}\op_1(b) \chi_1^{(+)}\|_{\GS_1}\le C_m\| \p^m b\|_{\plainL1}
\le C_m (\a\ell)^{1-m}\| \p^m a\|_{\plainL1},
\end{equation*}
for any $m\ge 3$. This is the required estimate. 
\end{proof}

\begin{remark}
Theorem 2.7 in \cite{Sob1} contains two misprints:
the number $n$ should be defined by the formula 
$n := \lceil 2q^{-1}\rceil+1$, 
and the main estimate of the 
Theorem should have the factor $r^{2q^{-1}-m}$ instead 
of $r^{q^{-1}-m}$. 
\end{remark}

Now we proceed to estimating trace norms of the operators $H_\a$, $Z_\a$ 
introduced above. 
Recall that $M^{(m)}(a, b)$ is defined in \eqref{M:eq}. 

\begin{lem}\label{Z:lem} Let $I=(x_0,y_0)$ with $y_0-x_0\ge 1$.  Then 
for $m\ge 3$ and any $\a > 0$ we have  
\begin{equation}\label{Z:eq}
\|Z_\a(a, b; I)\|_{\GS_1} \le  C_m \a^{-m+1} 
M^{(m)}(a, b),
\end{equation}
and 
\begin{align}
\| H_\a(a, b; \R_{y_0}^{(-)})&\  H_\a(a, b; \R_{x_0}^{(+)})\|_{\GS_1}
+ \| H_\a(a, b; \R_{x_0}^{(+)}) 
H_\a(a, b; \R_{y_0}^{(-)})\|_{\GS_1}\notag\\
\le &\  
C_m \a^{-m+1} \|a\|_{\plainL\infty} 
\|b\|_{\plainL\infty}M^{(m)}(a, b).
\label{HH:eq}
%
\end{align}
\end{lem}

\begin{proof}
We denote $A := \op_\a(a), B := \op_\a(b)$. 
Clearly, the operator $Z := Z_\a(a, b; I)$ splits into the sum
\begin{align*}
Z = &\ Z^{(1)} + Z^{(2)},\\[0.2cm]
Z^{(1)} := &\ \chi_I A 
\chi_{y_0}^{(+)} B \chi_{I} 
-  \chi_{y_0}^{(-)} 
A \chi_{y_0}^{(+)} B \chi_{y_0}^{(-)}, \\[0.2cm]
Z^{(2)} := &\ \chi_I 
A \chi_{x_0}^{(-)} B \chi_I
- \chi_{x_0}^{(+)} A  
\chi_{x_0}^{(-)} B  \chi_{x_0}^{(+)}.
\end{align*}
Let us rewrite
\begin{align*} 
Z^{(1)} = &\ -\chi_{x_0}^{(-)} A 
\chi_{y_0}^{(+)} B \chi_I 
-  \chi_{y_0}^{(-)} A  
\chi_{y_0}^{(+)}B  \chi_{x_0}^{(-)}, \\[0.2cm]
Z^{(2)} = &\ - \chi_{y_0}^{(+)}
A \chi_{x_0}^{(-)} B \chi_I
- \chi_{x_0}^{(+)} A  \chi_{x_0}^{(-)}
B \chi_{y_0}^{(+)}.
\end{align*}
Then, by Proposition \ref{separate:prop},
\begin{align*} 
\|Z^{(1)}\|_{\GS_1} 
\le &\ \|\chi_{x_0}^{(-)} A 
\chi_{y_0}^{(+)} \|_{\GS_1} \| B\| 
+ \|A\| 
\ \| \chi_{y_0}^{(+)}B  \chi_{x_0}^{(-)}\|_{\GS_1}\\[0.2cm]
\le &\ C_m \a^{-m+1} \bigl(\|\p_\xi^m a\|_{\plainL1} \|b\|_{\plainL\infty} 
+ \|a\|_{\plainL\infty}\|\p_\xi^m b\|_{\plainL1}\bigr)
=  C_m \a^{-m+1} M^{(m)}(a, b).
\end{align*}
Adding it up with the same bound for 
the operator $Z^{(2)}$ 
completes the proof of \eqref{Z:eq} for $Z(a, b; I)$.

Proof of \eqref{HH:eq}:   Let $z_0 := (x_0+y_0)/2$, so that the trace norm of 
the operator
\begin{equation*}
H_\a(a, b; \R_{y_0}^{(-)}) H_\a(a, b; \R_{x_0}^{(+)})
= \chi_{y_0}^{(-)} A \chi_{y_0}^{(+)} B \chi_I
A \chi_{x_0}^{(-)} B \chi_{x_0}^{(+)}
\end{equation*}
can be estimated by
\begin{equation*}
\| A\|^2 \ \|B\| \ \|
\chi_{y_0}^{(+)} B \chi_{(x_0, z_0)} \|_{\GS_1} 
 + \|A\|\ \| B \|^2\ \|\chi_{(z_0, y_0)}
A \chi_{x_0}^{(-)}\|_{\GS_1}.
\end{equation*}
Now Proposition \ref{separate:prop} leads to 
the bound \eqref{HH:eq} for the first term on the left-hand side of \eqref{HH:eq}.
In the same way one proves the same bound for the second term  
on the left-hand side. 
\end{proof}

\begin{lem}\label{TT:lem}
Let the conditions of Lemma \ref{Z:lem} be satisfied, 
and let $g\in\plainC{m}(\R)$ be another symbol. Then 
\begin{align}\label{TT1:eq}
\big\| \bigl[W_\a(a; I) - W_{\a}(a; &\ \R_{x_0}^{(+)})\bigr]
H_\a(b, g; \R_{x_0}^{(+)})\big\|_{\GS_1}\notag \\[0.2cm]
+ \big\| & \bigl[W_\a(a; I) - W_{\a}(a; \R_{y_0}^{(-)})\bigr]
H_\a(b, g; \R_{y_0}^{(-)})\big\|_{\GS_1}\notag\\[0.2cm] 
&\ \quad\quad\quad \quad 
\le C_m \a^{-m+1} \|g\|_{\plainL\infty}M^{(m)}(a, b).
\end{align} 
\end{lem}

\begin{proof}
With $A := \op_\a(a), B := \op_\a(b)$ we write
\begin{equation}\label{TT3:eq}
W_{\a}(a; \R_{x_0}^{(+)}) - W_\a(a; I)
= \chi_I A \chi_{y_0}^{(+)} 
+ \chi_{y_0}^{(+)} A\chi_{x_0}^{(+)}
\end{equation}
and estimate
\begin{equation*}
\|\chi_I A \chi_{y_0}^{(+)} H_\a(b, g; \R_{x_0}^{(+)})\|_{\GS_1}
\le \|a\|_{\plainL\infty} \|g\|_{\plainL\infty}
\|\chi_{y_0}^{(+)} 
 B \chi_{x_0}^{(-)}\|_{\GS_1}.
\end{equation*}
For the second term on the right-hand side of \eqref{TT3:eq} 
let $z_0 := (x_0+y_0)/2$. Then
\begin{align*}
\| \chi_{y_0}^{(+)} A\chi_{x_0}^{(+)}  
&\ H_\a(b, g; \R_{x_0}^{(+)}\|_{\GS_1} \\
\le &\ \|\chi_{y_0}^{(+)} A\chi_{z_0}^{(-)}\|_{\GS_1} \|b\|_{\plainL\infty} 
\|g\|_{\plainL\infty} 
+ \|a\|_{\plainL\infty} \|g\|_{\plainL\infty}
\|\chi_{z_0}^{(+)} B \chi_{x_0}^{(-)}\|_{\GS_1}.
\end{align*}
Now Proposition \ref{separate:prop} leads to inequality \eqref{TT1:eq} for 
the first term on the left-hand side of \eqref{TT1:eq}. 
The remaining inequality is derived in the same way. 
\end{proof}
 
\subsection{Estimates for $D_\a(a, I; r_z)$: one-dimensional case}
We apply definition \eqref{halpha:eq} to the symbols $ a_z:=a-z$ and $a_z^{-1}$. 
Now we assume that $a$ is a real-valued symbol and $\d(z, a) >0$, 
see \eqref{distz:eq} for definition. Thus we obtain 
(replace $\L$ by $I$)
 \begin{equation*}
 \chi_I -  W_\a(a_z; I) W_\a\bigl(a_z^{-1}; I\bigr) 
 = H_\a(a_z, a_z^{-1}; I).
 \end{equation*}
Clearly, both operators $W_\a(a_z; I)$ and $W_\a(a_z^{-1}; I)$ 
are invertible on $\plainL2(I)$ and 
\begin{equation*}
\bigl\| \left.W_\a(a_z; I)\right|_I^{-1}\bigr\|\le \frac{1}{\d(z, a)},\ 
\bigl\|\left.W_\a(a_z^{-1}; I)\right|_I^{-1}\bigr\|
\le |z|+ \|a\|_{\plainL\infty}.
\end{equation*} 
Thus  $\mathds{1}-H_\a(a_z, a_z^{-1}; I)$ is invertible on $\plainL2(\R)$ and 
\begin{equation*}
\bigl(\mathds{1}-H_\a(a_z, a_z^{-1}; I)\bigr)^{-1} \chi_I
= \left.W_\a(a_z^{-1}; I)\right|_I^{-1}
\left.W_\a(a_z; I)\right|_I^{-1},
\end{equation*}
with the bound
\begin{equation}\label{norm_bound_h:eq}
\|\bigl(\mathds{1}-H_\a(a_z, a_z^{-1}; I)\bigr)^{-1} \|
\le \frac{ |z|+ \|a\|_{\plainL\infty}}{\d(z, a)}.
\end{equation}
As a consequence,
\begin{align}
\bigl(W_\a(a; I) - z\bigr)^{-1} \chi_I - &\ W_\a\bigl(a_z^{-1};I\bigr)\notag\\[0.2cm]
=  &\ W_\a\bigl(a_z^{-1};I\bigr) 
\biggl[
\left.W_\a(a_z^{-1}; I)\right|_I^{-1}
\left.W_\a(a_z; I)\right|_I^{-1} - \chi_I
\biggr] \notag\\[0.2cm]
= &\ W_\a\bigl(a_z^{-1};I\bigr) H_\a(a_z, a_z^{-1}; I)
\bigl[\mathds{1}-H_\a(a_z, a_z^{-1}; I)\bigr]^{-1}.\label{W_rep:eq}
\end{align} 
Let us analyse the part of the right-hand side which contains $H_a$.

\begin{lem}\label{HH:lem}
Let $I = (x_0, y_0)$, and let $y_0-x_0\ge 1$.
Denote
\begin{equation*}
 H_\a := H_\a(a_z, a_z^{-1}; I), \ 
H^{(1)}_\a := H_\a(a_z, a_z^{-1}; \R_{x_0}^{(+)}), \ 
H^{(2)}_\a := H_\a(a_z, a_z^{-1}; \R_{y_0}^{(-)}). 
\end{equation*}
Then for any $\a>0$ and any $m\ge 3$,
\begin{align*}
\big\| H_\a&\ (\mathds{1} - H_\a)^{-1} 
- H^{(1)}_\a(\mathds{1} - H^{(1)}_\a)^{-1}
- H^{(2)}_\a(\mathds{1} - H^{(2)}_\a)^{-1}\big\|_{\GS_1}\\
\le &\ 
C_m \a^{-m+1} \biggl(\frac{|z| + \| a\|_{\plainL\infty}}{\d(z, a)}\biggr)^3 
 M^{(m)}(a_z, a_z^{-1}).
\end{align*}
\end{lem}

\begin{proof} 
We use the representation 
\begin{equation*}
H_\a = H^{(1)}_\a + H^{(2)}_\a 
+ Z_\a, \ Z_\a := Z_\a(a_z, a_z^{-1}; I).
\end{equation*}
The required bound for 
$ Z_\a(1-H_\a)^{-1}$  
follows from \eqref{Z:eq} 
and \eqref{norm_bound_h:eq}. 
Now, by the resolvent identity, 
\begin{align*}
\big\|H_\a^{(1)}(\mathds{1}-H_\a)^{-1} 
- &\ H_\a^{(1)}(\mathds{1}-H^{(1)}_\a)^{-1}\big\|_{\GS_1}
\le \|(\mathds{1}-H^{(1)}_\a)^{-1}\|\\[0.2cm]
 &\ \times \biggl[\|H^{(1)}_\a H^{(2)}_\a \|_{\GS_1}
+ \|H_\a^{(1)}\|\ \|Z_\a\|_{\GS_1} \biggr]
   \|(\mathds{1}-H_\a)^{-1}\|.
\end{align*}
The required bound for this operator follows from \eqref{norm_bound_h:eq}, 
\eqref{Z:eq}  and  \eqref{HH:eq}. 
 \end{proof}

\begin{lem} \label{red_resolv:lem}
For $z\in\mathbb C$ let $g$ be the function defined as $g(\l) := r_z(\l) := (\l -z)^{-1}$ for
$\l\in\R$. Also, with the symbol $a$ as above, let $a_z := a - z$. Then for any $\a>0$, 
\begin{align*}
\bigl\| D_\a(a, I; r_z) - 
D_\a(a, \R_{y_0}^{(-)}; r_z) 
- &\ D_\a(a, \R_{x_0}^{(+)}; r_z)\bigr\|_{\GS_1}\\[0.2cm]
\le  &\ C_m  
\a^{-m+1} \frac{1}{ \d(z, a)}
\biggl(\frac{|z|+\|a\|_{\plainL\infty}}{\d(z, a)}\biggr)^3 M^{(m)}(a_z, a_z^{-1}).
\end{align*}
\end{lem}

\begin{proof} We use the notation $H_\a, H^{(1)}_\a, H^{(2)}_\a$
 from Lemma \ref{HH:lem}, and  
\begin{equation*}
 W_\a :=  W_\a(a_z^{-1}; I), \ 
W^{(1)}_\a := W_\a(a_z^{-1}; \R_{x_0}^{(+)}), \ 
W^{(2)}_\a := W_\a(a_z^{-1};  \R_{y_0}^{(-)}). 
\end{equation*}
By Lemma \ref{TT:lem} and the bound \eqref{norm_bound_h:eq},
\begin{equation*}
\|(W_\a - W^{(k)}_\a) 
H^{(k)}_\a\bigl(\mathds{1}-H^{(k)}_\a\bigr)^{-1}\|_{\GS_1}
\le  C \a^{-m+1} 
\frac{|z|+\|a\|_{\plainL\infty}}{\d(z, a)^2} M^{(m)}(a_z, a_z^{-1}),
\end{equation*}
for  $k = 1, 2$. Together with  
Lemma \ref{HH:lem} this gives
\begin{align*}
\|W_\a H_\a(\mathds{1}-H_\a)^{-1}
- &\ \sum_{k=1}^2 W^{(k)}_\a 
H^{(k)}_\a(\mathds{1}-H^{(k)}_\a)^{-1}\|_{\GS_1}\\[0.2cm]
\le &\  C \a^{-m+1} \frac{1}{\d(z, a)}
\biggl(\frac{|z|+\|a\|_{\plainL\infty}}{\d(z, a)}\biggr)^3 M^{(m)}(a_z, a_z^{-1}), j = 1, 2.
\end{align*}
Now formula \eqref{W_rep:eq} leads to the proclaimed estimate. 
\end{proof}

\begin{proof}[Proof of Theorem \ref{interval_res:thm} for the case $I=(x_0, y_0)$] 
Lemma \ref{red_resolv:lem} shows that for
the function $r_z$ defined as $r_z(\l) := (\l-z)^{-1}$ the trace of
$D_\a(a, I; r_z)$ coincides with the sum 
\begin{equation}\label{halfax:eq}
\tr D_\a(a, \R_{y_0}^{(-)}; r_z) + \tr D_\a(a, \R_{x_0}^{(+)}; r_z)
\end{equation}
up to the remainder specified in the lemma. 
As we have observed earlier, due to the translation and reflection invariance, 
each of the intervals $\R_{x_0}^{(+)}, \R_{y_0}^{(-)}$ 
in the above trace sum can be replaced by $(0, \infty)$. 
When calculating the traces in \eqref{halfax:eq}, 
by making the 
change of variables $x\to \a x$ we can take $\a = 1$. Now Theorem 
\ref{interval_res:thm} follows from Proposition \ref{Widom_82:prop}(1). 
\end{proof}

\section{Proof of Theorem \ref{interval_res:thm}: 
the case of multiple intervals}\label{multiple:sect}

In this section we consider general sets $I$ of the form \eqref{Kint:eq}, and 
assume that \eqref{mult_int:eq} is satisfied. Throughout this section 
we assume that $a\in \plainC{m}(\R)$ with some $m\ge 3$ and that $a$ 
is real-valued. The parameter $\a$ is allowed to take any 
positive value and the constants 
in all estimates obtained are independent 
of the function $a$ or the parameters $z$ with $\d(z, a) >0$ and $\a$.   
Our strategy is to reduce the case of general $I$'s either to the 
case of one bounded interval, considered in the previous section, 
or to the case of the half-line, covered by Proposition \ref{Widom_82:prop}. 
More precisely, our objective is to prove the following result:

\begin{thm}\label{addit:thm}
For all $\a >0$ we have 
\begin{align}\label{addit:eq}
\|D_\a(a, I; r_z) - &\ \sum_{k} D_\a(a, I_k; r_z) \|_{\GS_1}
\notag\\[0.2cm]
\le  &\ C \a^{-m+1}  \frac{1}{\d(z, a)} 
\biggl(\frac{|z|+\|a\|_{\plainL\infty}}{\d(z, a)}\biggr)^3 
M^{(m)}(a_z, a_z^{-1}),
\end{align}
where $M^{(m)}(a,b)$ is defined in \eqref{M:eq}.
\end{thm}

The proof consists of several steps:

\begin{lem}Under the above conditions
\begin{equation}\label{s_mnogo:eq}
\| W_\a(a_z^{-1}; I) - \sum_k W_\a(a_z^{-1}; I_k) \|_{\GS_1}
\le C \a^{-m+1} 
\frac{1}{\d(z, a)}M^{(m)}(a_z, a_z^{-1}),
\end{equation}
and with $H_\a(a,b;I)$ defined in \eqref{halpha:eq},
\begin{equation}\label{h_mnogo:eq}
\|H_\a(a_z, a_z^{-1}; I) - \sum_k H_\a(a_z, a_z^{-1}; I_k)\|_{\GS_1}
\le C\a^{-m+1} M^{(m)}(a_z, a_z^{-1}).
\end{equation}

\end{lem}

\begin{proof} In order to prove \eqref{s_mnogo:eq} we write
\begin{equation*}
W_\a(a_z^{-1}; I) - \sum_k W_\a(a_z^{-1}; I_k)
= \sum_{j, k: j\not = k} \chi_{I_k} \op_\a(a_z^{-1}) \chi_{I_j}.
\end{equation*}
Due to the condition \eqref{mult_int:eq}, 
by Proposition \ref{separate:prop}, the trace norm of the right-hand side 
does not exceed 
\[
\a^{-m+1} \|\p_\xi^m a_z^{-1}\|_{\plainL1}\le \a^{-m+1} 
\frac{1}{\d(z, a)}M^{(m)}(a_z, a_z^{-1}), 
\]
as required. For the proof of \eqref{h_mnogo:eq} we write
\begin{equation*}
H_\a(a_z, a_z^{-1}; I) - \sum_k H_\a(a_z, a_z^{-1}; I_k)
= -\sum \chi_{I_k} \op_\a(a_z) \chi_{I_j} \op_\a(a_z^{-1}) \chi_{I_s},
\end{equation*}
where the sum is taken over the indices such that either $j\not =k$ or $j\not =s$. 
By Proposition \ref{separate:prop}, the trace norm of the right-hand side does not exceed
\begin{align*}
\a^{\-m+1}
\bigl(\|a_z\|_{\plainL\infty} \|\p_\xi^m a_z^{-1}\|_{\plainL1} 
+ \|a_z^{-1}\|_{\plainL\infty} \|\p_\xi^m a_z\|_{\plainL1} 
\bigr) = \a^{-m+1} M^{(m)}(a_z, a_z^{-1}),
\end{align*}
as required.
\end{proof}

\begin{lem}\label{H_split:lem}Under the above conditions
\begin{align*}
\big\|\bigl[ \mathds{1} - &\ H_\a(a_z, a_z^{-1}; I) \bigr]^{-1}H_\a(a_z, a_z^{-1}; I) -
\sum_k \bigl[ \mathds{1} - H_\a(a_z, a_z^{-1}; I_k)\bigr]^{-1}H_\a(a_z, a_z^{-1}; I_k)\big\|_{\GS_1}
\\
&\le C \a^{-m+1} \biggl(\frac{|z|+\|a\|_{\plainL\infty}}{\d(z, a)}\biggr)^3 M^{(m)}(a_z, a_z^{-1}).
\end{align*}
\end{lem}

\begin{proof}
For brevity we denote 
$H_\a := H_\a(a_z, a_z^{-1}; I)$, 
$H_\a^{(k)} := H_\a(a_z, a_z^{-1}; I_k)$. 
Due to \eqref{norm_bound_h:eq} and 
\eqref{h_mnogo:eq}, in the first term we can replace 
$H_\a$ with $\sum_{k} H_\a^{(k)}$. 
Now we estimate, using the resolvent identity:
\begin{align*}
\big\|(\mathds{1}-H_\a)^{-1} H_\a^{(k)} 
- &\ (\mathds{1}-H_\a^{(k)})^{-1} H_\a^{(k)}\big\|_{\GS_1}\\[0.2cm]
\le &\ \|(\mathds{1}-H_\a)^{-1}\| \ 
\|(\mathds{1}-H_\a^{(k)})^{-1}\|\ 
\|( H_\a - H_\a^{(k)})H_\a^{(k)}\|_{\GS_1}\\[0.2cm]
\le &\ \biggl( 
\frac{|z|+\|a\|_{\plainL\infty}}{\d(z, a)}\biggr)^2 
\biggl[
\|H_\a - \sum_{j} H_\a^{(j)}\|_{\GS_1}\|H_\a^{(k)}\| 
+ \sum_{j\not =k} \|H_\a^{(j)} H_\a^{(k)}\|_{\GS_1}
\biggr],
\end{align*}
where we have used \eqref{norm_bound_h:eq} again. 
As $j\not = k$ in the last term in the square brackets, this term equals zero. 
Now the required bound follows from \eqref{h_mnogo:eq} and the bound 
\begin{equation}\label{H_b:eq}
\|H_\a^{(k)}\|\le 
\frac{|z|+\|a\|_{\plainL\infty}}{\d(z, a)}.
\end{equation}
\end{proof}

\begin{proof}[Proof of Theorem \ref{addit:thm}]
As in the previous proof we use the notation $H_\a, H_\a^{(k)}$. 
Also, we denote $W_\a := W_\a(a_z^{-1}; I)$, $W_\a^{(k)} := W_\a(a_z^{-1}; I_k)$. 
 In view of \eqref{W_rep:eq},
\begin{align*}
\|D_\a(a, I; r_z) - &\ \sum_k D_\a(a, I_k; r_z)\|_{\GS_1}\\[0.2cm]
\le &\ \bigl\| W_\a - \sum_k W_\a^{(k)}\bigr\|_{\GS_1} \|(\mathds{1}-H_\a)^{-1} H_\a\|\\[0.2cm]
&\ + \sum_{k} \|W_\a^{(k)} \|\ \bigl\|(\mathds{1}-H_\a)^{-1} H_\a 
- \sum_k (\mathds{1}-H_\a^{(k)})^{-1} H_\a^{(k)}\bigr\|_{\GS_1}.
\end{align*}
The first term on the right-hand side satisfies 
\eqref{addit:eq} by \eqref{s_mnogo:eq} and \eqref{norm_bound_h:eq}, \eqref{H_b:eq}. 
The second term satisfies \eqref{addit:eq} by Lemma \ref{H_split:lem} and due to the bound
$\|W_\a^{(k)}\|\le \d(z, a)^{-1}$.
\end{proof}

\begin{proof}[Proof of Theorem \ref{interval_res:thm}]
By \eqref{addit:eq}, it remains to use the results for individual 
operators $D_\a(a, I_k; r_z)$. For $k = 1, 2, \dots, K$, that is, when 
$I_k$ is a bounded interval, we use the bound \eqref{int_res:eq} proved 
previously. If $k = 0$ or $k = K+1$, that is, when $I_k$ is a half-line, 
we use the identity \eqref{Widom_82:eq}. This leads to 
the bound \eqref{int_res:eq}, and completes the proof of Theorem 
\ref{interval_res:thm}.
\end{proof}

\section{Estimates for $D_\a(a, \L; f)$ 
with Fermi symbol $a=a_{T,\mu}$: multi-dimensional case}
\label{gen_bounds:sect}

As explained in the Introduction, 
the asymptotic analysis in this paper was partly motivated by the 
study of the entanglement entropy of free fermions. 
Thus in this section we 
apply the results obtained above to 
the special choice of the symbol $a$ featuring in definition 
\eqref{Dalpha:eq}. We choose the symbol $a$ 
to be the Fermi symbol $a_{T, \mu}$ defined in \eqref{positiveT:eq}.  
The choice of the (non-smooth) 
function $f$ remains arbitrary for the time-being. 
Further on, in Section \ref{entropy:sect}, we 
specialise to the R\'enyi 
entropy function $f=\eta_\gamma$, $\g >0$.

The physical context of the various quantities is as follows.
We assume that the energy of a single particle in position space $\R^d$ 
consists only of kinetic energy in the absense of external forces
and is determined by a Hamiltonian 
$h = h(\bxi)$ and that, for simplicity, 
particles do not have a spin-degree of freedom. 
The free Fermi gas is then a collection of 
infinitely many such particles 
obeying the (Pauli--)Fermi--Dirac statistics. 
An equilibrium state of this free 
Fermi gas is uniquely determined by 
specifying the temperature $T>0$, the chemical potential $\mu\in\R$, 
and the Fermi symbol \eqref{positiveT:eq}.
%
%
%
%
We will assume that $\mu$ is fixed and $T\in (0, T_0]$, and, in particular, 
$T$ is allowed to become small, that is, $T\downarrow 0$. 
Our aim is to find estimates with explicit dependence on $T$ and $\a$.

In what follows it will be convenient to use the following notation. 
For any two non-negative 
functions $x$ and $y$ depending on all or some of the variables/parameters 
$\a, T, \bxi$, we write $x\asymp y$ if there exist two constants $C, c$ independent 
of $\a, T, \bxi$ such that $cy\le x\le Cy$.

The assumptions on the function $h = h(\bxi)$ are as follows:

\begin{cond} \label{h:cond}
\begin{enumerate}
\item
The function $h\in\plainC\infty(\R^d)$ is real-valued, and for sufficiently large $\bxi$ 
and with some constants $\b_1>0$ and $c>0$ we have 
\begin{equation}\label{equih:eq}
h(\bxi)\ge c|\bxi|^{\b_1}. 
\end{equation}
Moreover, for some $\b_2\ge 0$
\begin{equation}\label{derivh:eq}
|\nabla^n h(\bxi)| \le C_n (1+|\bxi|)^{\b_2}, \ n= 1, 2, \dots,\ \ \bxi\in\R^d.
\end{equation}

\item  
On the set $S := \{\bxi\in\R^d: h(\bxi) = \mu\}$ the condition 
\begin{equation}\label{nondeg:eq}
\nabla h(\bxi)\not = 0,\ \bxi\in S
\end{equation}
is satisfied. 
\item
The Fermi sea $\Om := \{\bxi\in\R^d: h(\bxi) < \mu\}$ has finitely many connected components. 
\end{enumerate}
\end{cond}

Let us record some useful inequalities for the symbol $a=a_{T,\mu}$ from \eqref{positiveT:eq}. 

\begin{lem}\label{hbounds:lem}
	Suppose that $0 < T \le T_0$. Then 
	\begin{equation*}
	|\nabla^n a(\bxi) |\le C_n a(\bxi)\bigl(1-a(\bxi)\bigr) (1+|\bxi|)^{n\b_2} T^{-n}, n =1, 2, \dots,
	\end{equation*}
	with constants $C_n$ depending on $T_0$, $\mu$, and the constants in \eqref{derivh:eq}.
\end{lem}

The proof is elementary and thus omitted. 

A straightforward calculation leads to the bounds
\begin{equation}\label{step:eq}
|a(\bxi) - \chi_\Om(\bxi)|\le \exp{\biggl(-\frac{|h(\bxi) - \mu|}{T}\biggr)},\ \bxi\in\Rd,
\end{equation}
and
\begin{equation}\label{a1-a:eq}
a(\bxi) \bigl(1-a(\bxi))\le \exp{\biggl(-\frac{|h(\bxi) - \mu|}{T}\biggr)},\ \bxi\in\Rd.
\end{equation} 
Our objective is to obtain the following estimate. 

\begin{thm}\label{entropy:thm}  
	Suppose that the function $f$ 
	satisfies Condition \ref{f:cond} with some $n\ge 2$ and $\g >0$,  
	Suppose also that the region
	$\L$ and the function $h$ 
	satisfy Conditions \ref{domain:cond} and 
	\ref{h:cond} respectively.  
	Let $\a T \ge \a_0>0, 0< T\le T_0$ 
	for some $\a_0$ and $T_0$. 
	Then for any $\s\in (0, \g)$, $\s \le 1$,  
	we have 
	\begin{equation}\label{entropy:eq}
	\bigl\| 
	D_\a(a_{T, \mu}, \L; f)\bigr\|_{\GS_1}
	\le C
	R^{\g-\s} \a^{d-1}(|\log(T)| + 1)  \1 f\1_n,
	\end{equation}
	with a constant $C$ independent of $T, R, t_0, \a$, 
	and the function $f$, 
	but depending on $\a_0, T_0, \mu$.
\end{thm}

Until the end of this section we always assume that the region $\L$ and the function $h$ 
satisfy Conditions \ref{domain:cond} and \ref{h:cond} respectively.

Because of \eqref{equih:eq} the set $\Om$ is 
bounded, so that $\Omega\subset B(\mathbf0, R_0)$ 
with some $R_0 >0$. Assume first that $d\ge 2$. 
Due to condition 
\eqref{nondeg:eq}, the set $S$ is locally a $\plainC\infty$-surface, 
which is called the \textit{Fermi surface}.
 More precisely, for any $\bxi_0\in S$ 
 there is a radius $r>0$ such that $|\p_{\xi_d}h(\bxi)|\ge c$ for 
 all $\bxi\in B(\bxi_0,2r)$ with a suitable choice of coordinates    
 $\bxi = (\hat\bxi, \xi_d)$, and hence there exists a 
 function $\Psi\in\plainC\infty(\R^{d-1})$ 
 such that  
 \begin{equation}\label{boundary:eq}
 S\cap B(\bxi_0, 2r)
 = \{\bxi\in\R^d: \xi_d = \Psi(\hat\bxi)\}\cap B(\bxi_0, 2r). 
 \end{equation}
 For definiteness we assume that $B(\bxi_0, 2r)
 \subset B(\mathbf0, R_0)$. 
 We may also assume that 
 \begin{equation}\label{telo:eq}
 \Om\cap B(\bxi_0, 2r) 
 = \{\bxi\in\R^d: \xi_d > \Psi(\hat\bxi)\}\cap B(\bxi_0, 2r).
\end{equation} 
 This can be achieved by replacing $\xi_d$ and 
 $\Psi(\hat\bxi)$ with $-\xi_d$ and $-\Psi(\hat\bxi)$ and by taking a smaller $r$, 
 if necessary. Without loss of generality we 
 may assume that $\|\nabla\Psi\|_{\plainL\infty}\le M$ 
 with  some constant $M>0$. 
 By choosing a sufficiently small $r>0$, due to the condition $|\p_{\xi_d}h|\ge c$ 
 one can also guarantee that 
\begin{equation}\label{twosided:eq}
|\xi_d-\Psi(\hat\bxi)|\asymp |h(\bxi)-\mu|,
\  \bxi\in B(\bxi_0, 2r),
 \end{equation}
 with some $C\ge 1$. 
It is clear that $|\xi_d - \Psi(\hat\bxi)|\ge \dist(\bxi, S)$. 
On the other hand,  
$|\bxi-\boldeta|\ge (1+M^2)^{-1/2} |\xi_d-\Psi(\hat\bxi)|$, 
for any $\bxi \in B(\bxi_0, 2r)$ and any $\boldeta\in S\cap B(\bxi_0, 2r)$. 
Consequently,
\begin{equation}\label{distvs:eq}
|\xi_d - \Psi(\hat\bxi)|\asymp \dist(\bxi, S),\ \forall \bxi\in B(\bxi_0, r).
\end{equation}
Since the set $\Omega$ is in fact a $\plainC\infty$-
region, we can cover its boundary $S$ with finitely many open balls 
$\{D_j(r)\}$ of radius $r$ centred at 
some $\bxi_j\in S$, such that in each $D_j(2r)$ one 
can find an appropriate function $\Psi=\Psi_j$ 
that satisfies the properties \eqref{boundary:eq}--\eqref{twosided:eq} 
after an appropriate choice of 
coordinates in every ball $D_j(2r)$. 
From now on for brevity we denote $D_j = D_j(r)$. 

Let $\tilde D\subset \R^d$ be a region such 
that $\tilde D\cap S = \varnothing$, and 
\begin{equation}\label{coverrd:eq}
\R^d = (\cup_j D_j) \cup\tilde D,\  \ 
\tilde D = (\cup_j \tilde D_j) \cup \{\bxi\in\R^d: |\bxi|>R_0\}.
\end{equation}  
If $d = 1$, then we modify the definitions of 
$\{D_j\}$ and $\tilde D$ in an obvious way. 
For example, each $D_j$ is now an interval such 
that with an appropriate choice of the 
coordinate $\xi$  the open set 
$D_j \cap \Om$ is simply $D_j\cap \{\xi\in\R: \xi >0\}$. Thus the covering 
\eqref{coverrd:eq} holds for $d = 1$ as well. 

The idea of the proof of Theorem \ref{entropy:thm} is to observe that the symbol 
\eqref{positiveT:eq} satisfies 
\eqref{scales:eq} on each element of the covering 
\eqref{coverrd:eq} with some functions $\tau$ and $v$ defined individually on each 
of the domains $D_j$ and $\tilde D$. After that Theorem \ref{multi:thm} 
produces Theorem \ref{entropy:thm}. 

Let us first describe the construction of the scaling function 
$\tau$ and amplitude function $v$ on $D_j$ and $\tilde D$. 
We do this for the case $d\ge 2$, 
as for $d=1$ only obvious modifications are required. 

Let $\Psi=\Psi^{(j)}\in\plainC\infty(\R^{d-1})$ be 
a function describing the surface $S$ inside $D_j$, 
see \eqref{boundary:eq}. Recall that we always assume that 
$\|\nabla\Psi\|_{\plainL\infty}\le C$. 
We introduce the functions 
$\ell^{(j)}$ and $w^{(j)}$ defined on $\R^d$ as
\begin{equation}\label{lj:eq}
\ell^{(j)}(\bxi) := |\xi_d - \Psi^{(j)}(\hat{\bxi})| + T,\ 
w^{(j)}(\bxi) := \exp\biggl(-c_1\dfrac{\ell^{(j)}(\bxi)}{T}\biggr) .
\end{equation}
Due to \eqref{step:eq}, \eqref{a1-a:eq} and \eqref{twosided:eq},  
the constant $c_1$ can be chosen to guarantee that 
\begin{equation}\label{a11:eq}
|a(\bxi) - \chi_\Om(\bxi)|
\le w^{(j)}(\bxi),\ \bxi\in D_j,
\end{equation}
and
\begin{equation}\label{a1:eq}
a(\bxi)\bigl(1-a(\bxi)\bigr)
\le w^{(j)}(\bxi),\ \bxi\in D_j.
\end{equation}
Since $D_j\subset B(\mathbf0, R_0)$, we get from Lemma \ref{hbounds:lem} that for $\bxi\in D_j$
\begin{equation*}
	|\nabla^n a(\bxi)|\le C_n T^{-n} a(\bxi)(1-a(\bxi))
	\le \tilde C_n T^{-n} w^{(j)}(\bxi), C_n = C_n(R_0). 
\end{equation*} 
Using the fact that $\sup_{t>0}t^n e^{-t}$ is finite for 
all $n = 0, 1, \dots, $ we can estimate the right-hand side by 
$C_n \ell^{(j)}(\bxi)^{-n}$. 
Therefore
\begin{equation}\label{ltau_0:eq}
|\nabla^n a(\bxi)|\le C_n \ell^{(j)}(\bxi)^{-n},
\ n = 0, 1, 2, \dots,\  \forall \bxi\in D_j. 
\end{equation}
This shows that on $D_j$ the symbol 
$a$ satisfies \eqref{scales:eq} with 
$\tau(\bxi) := \ell^{(j)}(\bxi)$ and $v(\bxi) := v^{(j)}(\bxi) = 1$. 
 
On the domain $\tilde D$ the construction is different. 
Define the function $\widetilde{w}$ as
\begin{equation}\label{l_tilde:eq}
\widetilde{w}(\bxi) := \exp\biggl(-c_1\frac{(1+|\bxi|)^{\b_1}}{T}
\biggr),\ \bxi\in\R^d. 
\end{equation}
Since 
$h$ satisfies \eqref{equih:eq}, and 
$|h(\bxi)-\mu|\ge c$ for $\bxi\in \tilde D $, 
one can find a constant $c_1>0$ such that 
\begin{equation*}
	\exp\biggl(
	-\frac{|h(\bxi)-\mu|}{T}
	\biggr)\le \widetilde{w}(\bxi)^2,\ \bxi\in \tilde D.
\end{equation*}
Hence, by \eqref{step:eq} and \eqref{a1-a:eq}, 
\begin{equation}\label{a21:eq}
|a(\bxi) - \chi_\Om(\bxi)|\le \widetilde{w}(\bxi)^2,\ \bxi\in\tilde D,
\end{equation}
and 
\begin{equation}\label{a2:eq}
a(\bxi)\bigl(1-a(\bxi)\bigr)\le \widetilde{w}(\bxi)^2,\ \bxi\in \tilde D.
\end{equation}
Consequently by Lemma \ref{hbounds:lem}, 
\begin{equation*}  
	|\nabla^n a(\bxi)|\le C_n T^{-n} (1+|\bxi|)^{\b_2 n} \widetilde{w}(\bxi)^2,
	%
	n = 1, 2, \dots,
\end{equation*}
for all $\bxi\in \tilde D$.  
Using the fact that $\sup_{t\ge 1}(t^{\b_2} T^{-1})^n e^{-c_1t^{\b_1}T^{-1}} \le C(n, T_0)$  
for all $T\in (0, T_0]$ and $n = 1, 2, \dots$, 
we conclude that 
\begin{equation*} 
	|\nabla^n a(\bxi)|\le  \tilde C_n \widetilde{w}(\bxi), \ n = 1, 2, \dots, \ \bxi\in\tilde D.
\end{equation*}
This implies that with a suitable constant $c_2 = c_2(T_0, h)$,
\begin{equation*}
	|\nabla^n  a(\bxi)|\le C_n 
	e^{-c_2|\bxi|^{\b_1}}, 
	\ n = 0, 1, \dots,\ \bxi\in\R^d.
\end{equation*}
It is more convenient to replace the exponential by a power-like function 
$\tilde v(\bxi) := \tilde v_\s(\bxi) := (1+|\bxi|)^{-(d+1)\s^{-1}}$ with some 
$\s\in (0, 1]$, so that for all $\bxi\in \tilde D$
\begin{equation}\label{ltautilde:eq}
|\nabla^n a(\bxi)|\le C_n (1+|\bxi|)^{-(d+1)\s^{-1}},\ n = 0, 1, \dots. 
\end{equation} 
Thus $a$ satisfies \eqref{scales:eq} with $\tilde\tau= 1$ and 
$v = \tilde v_\s$. The choice of the value of $\s$ will be made later. 

Now we can put together the definitions of $\ell^{(j)}, v^{(j)}$ and 
$\tilde\tau, \tilde v$ to define the scaling function and amplitude on the entire 
space. Let $\{\phi_j\}, \tilde\phi$ be a partition of 
unity subordinate to the covering 
\eqref{coverrd:eq}. Then we define for $\bxi\in\R^d$
\begin{equation}\label{tau_global:eq}
\begin{cases}
\tau(\bxi) := \theta\bigl(\sum_j \ell^{(j)}(\bxi) 
\phi_j(\bxi) + \tilde\phi(\bxi)\bigr),\\[0.2cm]
v(\bxi) := \sum_{j} \phi_j(\bxi) + (1+|\bxi|)^{-\frac{d+1}{\s}}\tilde\phi(\bxi).
\end{cases}
\end{equation} 
The constant $\theta >0$ is chosen to guarantee the bound 
$\|\nabla\tau\|_{\plainL\infty}\le \nu$ with some $\nu\in (0, 1)$. 
It is straightforward to check that $v$ satisfies \eqref{w:eq} and
that $\tau\asymp 1$ on $\tilde D$. 
 Moreover, by virtue of \eqref{distvs:eq}, $\tau\asymp \ell^{(j)}$ 
on $D_j$. Consequently, the symbol $a$ satisfies 
\eqref{scales:eq} with the functions $\tau$ and $v$ defined above.

Let us establish some bounds for $V_{\s, \rho}(v, \tau)$, see \eqref{Vsigma:eq}.  

\begin{lem}\label{Vso:lem}
Let $T\in (0, T_0]$. Let $\tau$ be defined as in \eqref{tau_global:eq} with 
the same $\s \in (0, 1]$ as in \eqref{tau_global:eq}. Then 
\begin{equation}\label{dj:eq}
V_{\s, 1}(v, \tau)\asymp \biggl|\log(T)\biggr| + 1,
\end{equation}
\begin{equation}\label{djm:eq}
V_{\s, \rho}(v, \tau)\le C_{\s, \rho} T^{-\rho+1},\ \rho > 1,
\end{equation}
with a constant independent of $T\in (0, T_0]$.
\end{lem}

\begin{proof}
We estimate integrals of the type \eqref{Vsigma:eq} over the domains 
that form the covering \eqref{coverrd:eq}. Denote
\begin{equation*}
V^{(j)}_{\s, \rho}(v, \tau) 
:= \int \phi_j(\bxi)\frac{v(\bxi)^\s}{\tau(\bxi)^\rho}
d\bxi,\ \ \ 
\tilde V_{\s, \rho}(v, \tau) 
:= \int \tilde\phi(\bxi)\frac{v(\bxi)^\s}{\tau(\bxi)^\rho}
d\bxi.
\end{equation*}
As we have observed previously, $\tau\asymp \ell^{(j)}$ on $D_j$, so that 
\begin{align*}
V^{(j)}_{\s, \rho}(v, \tau)
\asymp &\ \underset{D_j}\int (|\xi_d-\Psi^{(j)}(\hat\bxi)| + T)^{-\rho} 
d\bxi\\[0.2cm]
	\asymp &\ r^{d-1}
		\int_{-2r}^{2r} (|t|+T)^{-\rho}dt.  
\end{align*}
This leads to \eqref{dj:eq} and \eqref{djm:eq} for $V^{(j)}_{\s, \rho}$. 
Furthermore, $\tau(\bxi)\asymp 1$ for all $\bxi\in \tilde D$. Therefore 
\begin{equation*}
\tilde V_{\s, \rho}(v, \tau)\asymp \int (1+|\bxi|)^{-d-1} d\bxi\le C, 
\end{equation*}
for any $\rho\in\R$. The obtained bounds together prove 
\eqref{dj:eq} and \eqref{djm:eq}.
\end{proof}

\begin{proof}[Proof of Theorem \ref{entropy:thm}]  
	We use \eqref{multi:eq} 
	with $q = 1$ and any $\s\in (0, \g)$, $\s\le 1$:
\begin{equation*}
\| D_\a(a, \L; f)\|_{\GS_1}\le C_{\s, \g} \a^{d-1}\1 f\1_n R^{\g-\s} 
V_{\s, 1}(v,\tau).
\end{equation*}
Now Lemma  \ref{Vso:lem} leads to \eqref{entropy:eq}. 
\end{proof}

The case of a homogeneous function $h$ deserves special attention since in this case one can 
explicitly control the dependence on the chemical potential $\mu$. 
We illustrate this with the example of 
the function $h(\bxi) = |\bxi|^2$. The parameter $\mu$ can be ``scaled out" with the help 
of the following formula:
\begin{equation*}
\op_\a(a_{T, \mu}) = \op_\nu (a_{T', 1}),\ T' := T\mu^{-1},\ \nu := \a\sqrt\mu,
\end{equation*}
so that $D_\a(a_{T, \mu}, \L; f) = D_\nu(a_{T', 1}, \L; f)$. 
Thus Theorem \ref{entropy:thm} leads to the following result.

\begin{thm}\label{entropy_spec:thm} 
	Suppose that $f$ satisfies Condition \ref{f:cond} with some $n\ge 2$ and $\g >0$, 
	and that the region $\L$ satisfies Condition \ref{domain:cond}. 
	Let $a = a_{T, \mu}$ be given by \eqref{positiveT:eq} with 
	$h(\bxi) = |\bxi|^2$, and let 
	$\a T\mu^{-1/2}\ge \a_0, 0<T\mu^{-1}\le T_0$.  
		Then for any $\s\in (0, \g)$, $\s \le 1$, 
	\begin{equation}\label{entropy_spec:eq}
	\bigl\| 
	D_\a(a_{T, \mu}, \L; f)\bigr\|_{\GS_1}
	\le C_\s 
	R^{\g-\s}  \1 f\1_n(\a\sqrt{\mu})^{d-1}(|\log (T\mu^{-1})| + 1),
	\end{equation}
	with a constant $C_\s$ independent of $R, \a, \mu$,
 and the function $f$, but depending on 
	$\a_0, T_0$ and $\g, \s$.
\end{thm}

The final result in this section is specific to dimension one. 

\begin{thm}\label{interval_lowT:thm} 
Let $I$ and $\omega$ be defined as in \eqref{Kint:eq} 
and \eqref{omega:eq} respectively, and let the constituent intervals $I_j$ satisfy 
\eqref{mult_int:eq}. Suppose $h$ satisfies 
Condition \ref{h:cond} and that $f$ satisfies Condition \ref{f:cond} with some $\g >0$, 
$t_0\in\R$ and $n=2$. Furthermore, suppose that $T\in (0,1/2]$ and $\a T\ge \a_0>0$. Then 
	\begin{equation}\label{fn_lowT:eq}
	\underset{\a\to\infty}\lim\ \frac{1}{|\log(T)|}
	\bigl(\tr D_\a(a_{T, \mu}, I; f) - \omega \CB(a_{T, \mu}; f)\bigr) = 0,
	\end{equation}
	uniformly in $t_0\in\R$. Moreover, for any  
	$T\in (0, 1/2]$,
\begin{equation}\label{bbound:eq}
|\CB(a_{T, \mu}; f)|\le C_{\g, \s}\1 f\1_2 \,|\log(T)|,
\end{equation}
	uniformly in $t_0\in\R$.
	\end{thm}

\begin{proof}[Proof of Theorem \ref{interval_lowT:thm}]
Define the scale and the amplitude as in \eqref{tau_global:eq}.
Then the $\log$-bound \eqref{dj:eq}, together with 
\eqref{coeffscales:eq} imply \eqref{bbound:eq}.

In order to prove \eqref{fn_lowT:eq} we use the asymptotics 
\eqref{scale_asymp:eq}. First we check the condition \eqref{higher:eq}. 
By \eqref{dj:eq} and \eqref{djm:eq},  
the left-hand side of \eqref{higher:eq} is estimated by 
\begin{equation*}
\a^{-m+1} \frac{V_{\s, m}(v, \tau)}
{V_{\s, 1}(v, \tau)}\le C(\a T)^{-m+1}\frac{1}{|\log(T)|}, 
\end{equation*}
and hence it tends to zero under the conditions $\a T \ge \a_0$ and $\a\to\infty$. 
As a result, the condition \eqref{higher:eq} is satisfied, and therefore 
one can use \eqref{scale_asymp:eq}, which leads to \eqref{fn_lowT:eq}, as required. 
\end{proof}

The above formulas hold for arbitrary $T$ satisfying the condition 
$\a T\ge \a_0$. If we assume additionally that $T\downarrow 0$, 
then the asymptotics \eqref{fn_lowT:eq} can be written in a more 
explicit form, thanks to the asymptotic formula for $\CB(a_{T, \mu}; f)$, 
$T\downarrow 0$, obtained in Theorem \ref{CBTheorem:thm}, which,  incidentally, confirms the sharpness of 
the estimate \eqref{bbound:eq}. Recall that according to 
Condition \ref{h:cond}, for $d=1$ the set $\Om$ is represented as 
\begin{align}\label{omj:eq}
\Om = \bigcup_{j=1}^N J_j, \ N < \infty,
\end{align} 
where $\{J_j\}$ are bounded open intervals such that their closures are pairwise disjoint.

\begin{cor}
Let the set $I$, number $\omega$ and the functions $h$, 
$f$ be as in Theorem \ref{interval_lowT:thm}. 
Suppose that $T\downarrow 0$ and $\a T\ge \a_0>0$. Then 
	\begin{equation} \label{fn_lowT1:eq}
	\tr D_\a(a_{T, \mu}, I; f) = 
	|\log(T)| \biggl(\frac{\omega N}{2\pi^2} U(1, 0; f) 
	+ o(1)\biggr),
	\end{equation}
	uniformly in $t_0\in\R$, where $N$ is as in \eqref{omj:eq}. 
\end{cor} 
 
\begin{proof}
The claimed asymptotics follows immediately from 
Theorems \ref{interval_lowT:thm} and \ref{CBTheorem:thm}. 
\end{proof} 

\section{Asymptotics of $\CB(a_{T, \mu}; f)$ as $T\downarrow 0$}\label{low T:sect}

Here we study the behaviour of 
$\CB(a_{T, \mu}; f)$ with the Fermi
symbol $a_{T, \mu}$ defined in 
\eqref{positiveT:eq} as $T\downarrow 0$. 
The number $N$ below is as in the representation \eqref{omj:eq}. 

\begin{thm}\label{CBTheorem:thm}
Let $a_{T, \mu}$ be as in \eqref{positiveT:eq}, and let $h$ satisfy 
Condition \ref{h:cond}. Suppose that $f$ satisfies Condition \ref{f:cond} 
with some $t_0\in\R$, $\g >0$ and some $R\le 1$. Then, as $T\downarrow 0$
\begin{equation}\label{CBas:eq}
\CB(a_{T, \mu}; f) = \frac{N}{2\pi^2} \ U(1, 0; f)\,|\log(T)| + O(1),
\end{equation}
with $U(1, 0; f)$ defined in \eqref{U:eq}.
\end{thm}

Let $\tau$ and $v$ be as defined in \eqref{tau_global:eq}, so that 
$\tau_{\textup{\tiny inf}} = \t T$. 
We study separately the integral $\CB^{(1)}$ defined in 
\eqref{cb1:eq} and 
\begin{equation}\label{cb2:eq}
\CB^{(2)}(a; f) := \frac{1}{8\pi^2}
\underset{|\xi_1-\xi_2|> \frac{\t T}{2}}\iint
\frac{U(a(\xi_1), a(\xi_2); f)}{|\xi_1-\xi_2|^2}
d\xi_1 d\xi_2,\ a = a_{T, \mu}.
\end{equation}
Using \eqref{sub:eq} and \eqref{djm:eq}, 
we conclude that for all $T\in (0, T_0]$, 
\begin{equation}\label{sub2:eq}
|\CB^{(1)}(a; f)|\le C.
\end{equation}
To study $\CB^{(2)}$ we intend to replace $a$ with the indicator function 
$\chi_\Om$. To this end we note the following properties of the function 
$f$, and as a result, of the integral \eqref{U:eq}.  
The bound \eqref{hol:eq} says that the function $f$ is H\"older:
\begin{equation*}
|f(t_1) - f(t_2)|\le 2 \1 f\1_1 |t_1-t_2|^\varkappa,\ 
\varkappa := \min\{1,\g\}. 
\end{equation*}  
An elementary calculation shows that for any $\mu\in (0, 1)$ and for any 
real $s_1, r_1, s_2, r_2$ (see \cite[Formulas (2.4) and (3.8)]{Sob_16}) we have 
\begin{align}\label{approx:eq}
	|U(s_1, s_2; f) - U(r_1, r_2; f)|\le &\ C\1 f\1_1 
	|\log(\mu)|\bigl(|s_1-r_1|^\varkappa + |s_2-r_2|^\varkappa\bigr) \notag\\[0.2cm]
	&\ + C\1 f\1_1 \mu^\varkappa
	\bigl(|s_1-s_2|^\varkappa + |r_1-r_2|^\varkappa\bigr).
\end{align}
This  leads to the following useful lemma. 

\begin{lem} Let $f$ be as above, and suppose that $|s_1-s_2|+ |r_1-r_2|\le C$. 
Then for any $\d \in [0, \varkappa)$ we have 
\begin{equation}\label{approx1:eq}
	|U(s_1, s_2; f) - U(r_1, r_2; f)|\le  C_\d\1 f\1_1 
	\bigl(|s_1-r_1|^\d + |s_2-r_2|^\d\bigr). 
\end{equation}
\end{lem}

\begin{proof}
It suffices to assume that $|s_1-r_1| + |s_2-r_2| < 1/2$. 
Now \eqref{approx1:eq} follows from \eqref{approx:eq} if one sets 
$\mu = |s_1-r_1| + |s_2-r_2|$. 
\end{proof}

\begin{lem}
Let the condition of Theorem \ref{CBTheorem:thm} 
be satisfied. Then
\begin{equation} \label{b2diff:eq}
|\CB^{(2)}(a; f) - \CB^{(2)}(\chi_\Om; f)| \le C \1 f\1_1, 
\end{equation}
uniformly in $T\in (0, T_0]$.
\end{lem}

\begin{proof}
In view of \eqref{approx1:eq}, 
\begin{align}\label{b2:eq}
|\CB^{(2)}(a; f) - \CB^{(2)}(\chi_\Om; f)|
\le &\ C_\d  \1 f\1_1	\underset{\frac{\t T}{2}<|\xi_1-\xi_2|}\iint  
	\frac{|a(\xi_1) - \chi_\Om(\xi_1)|^\d}{|\xi_1-\xi_2|^2}
	d\xi_1 d\xi_2\notag\\[0.2cm]
	\le &\ C_\d T^{-1} \1 f\1_1\int  
	|a(\xi) - \chi_\Om(\xi)|^\d d\xi,
\end{align}
for any $\d\in [0, \varkappa)$. To estimate this integral we use 
a partition of unity subordinate to the covering 
\eqref{coverrd:eq}, as in Section \ref{gen_bounds:sect}.
Thus, in view of \eqref{step:eq} and \eqref{twosided:eq}, for each $D_j$ we obtain
\begin{equation*}
\underset{D_j}\int 
|a(\xi) - \chi_\Om(\xi)|^\d d\xi_1
\le C\int e^{-c\d\frac{|\xi|}{T}}d\xi\le C_\d T.
\end{equation*}
Also, since the set $\tilde D$ is separated from $\Om$, 
we have 
\begin{equation*}
|a(\xi_1) - \chi_\Om(\xi_1)| \le C \exp\biggl({-\frac{(1+|\xi_1|)^{\b_1}}{T}}\biggr),\ \xi\in\tilde D,
\end{equation*}
and hence
\begin{equation*}
\underset{\tilde D}\int 
|a(\xi_1) - \chi_\Om(\xi_1)|^\d d\xi_1\le C\int e^{-c\d\frac{|\xi_1+1|^{\b_1}}{T}}d\xi_1\le C_\d e^{-\frac{c}{T}}.
\end{equation*}
Together with \eqref{b2:eq}, the above estimates lead to \eqref{b2diff:eq}
\end{proof}

It remains to calculate $\CB^{(2)}(\chi_\Om; f)$. 
Since $U(1, 1; f) = U(0, 0; f) = 0$ and $U(1, 0; f) = U(0, 1; f)$, this coefficient reduces to
 \begin{equation*}
 \CB^{(2)}(\chi_\Om; f) = \frac{U(1, 0; f)}{4\pi^2}
	\underset{\xi_1\notin\Om}\int\ \ 
	\underset{\xi_2\in\Om: \frac{\t T}{2}<|\xi_1-\xi_2|}\int  
	\frac{1}{|\xi_1-\xi_2|^2}
	d\xi_2 d\xi_1.
 \end{equation*}

The next lemma seems to be useful in its own right, 
where we claim a certain uniformity in the size of the intervals $J_k$, 
$k = 1, 2, \dots, N$, although Theorem \ref{CBTheorem:thm} does not 
need this. 
 
\begin{lem} \label{log_int:lem}
Let $J_k = (s_k, t_k)\subset \R$, $k = 1, 2, \dots, N$ 
be a finite collection of bounded open intervals, 
such that their closures are pairwise disjoint, and let 
$J = \cup_k J_k$. Suppose that $T\in (0, T_0]$ and 
$|J_k|\le d_1$, $k = 1, 2, \dots, N$, with some $d_1> 0$.  
Then 
\begin{equation}\label{log_bound:eq}
\sum_{k=1}^N\underset{t\notin J}
\int dt\,\underset{|t-s|\ge T, s\in J_k}
\int \frac{ds}{|t-s|^2} \le CN\,|\log(T)|,
\end{equation}
with a constant $C$ depending only on $d_1$. 

Assume in addition that 
\begin{equation*}
|J_k|\ge d_0,\ k = 1, 2, \dots, N,\ \ 
\ \min_{j\not=k}\dist\{J_k, J_j\}\ge d_0.
\end{equation*}
with some $d_0\in (0, d_1]$. 
Let $\varphi\in \plainC{}(\R)\cap \plainL\infty(\R)$ be a 
function. 
Then, as $T\downarrow 0$,
\begin{align}\label{asb:eq}
\sum_{k=1}^N\underset{t\notin J}
\int \varphi(t)dt&\ \underset{|t-s|\ge T, s\in J_k}
\int \frac{ds}{|t-s|^2}\notag\\[0.2cm]
 =  &\ |\log(T)|\sum_{k=1}^N
 \bigl(\varphi(s_k) 
 + \varphi(t_k)\bigr) + N \|\varphi\|_{\plainL\infty}O(1),
\end{align}
where $O(1)$ depends only on $d_0$ and $d_1$.
\end{lem}

\begin{proof}
Proof of \eqref{asb:eq}. 
Without loss of generality assume that $\|\varphi\|_{\plainL\infty} = 1$.  
It is immediate to see that 
$$
\sum_{k=1}^N\underset{t\notin J}\int \varphi(t) dt\ \underset{|t-s|\ge T, s\in J_k}
\int \frac{ds}{|t-s|^2} = \sum_{k=1}^N\underset{t\notin J_k}\int \varphi(t)dt
\underset{|t-s|\ge T, s\in J_k}\int \frac{ds}{|t-s|^2} + N O(1)
$$
so that \eqref{asb:eq} reduces to showing that
\begin{equation*}
\sum_{k=1}^N\underset{t\notin J_k}\int \varphi(t) dt
\underset{|t-s|\ge T, s\in J_k}\int \frac{ds}{|t-s|^2}
 =  |\log(T)| \sum_{k=1}^N
\bigl(\varphi(s_k) + \varphi(t_k)\bigr) + N O(1), \ T\downarrow 0.
\end{equation*}
Hence it suffices to prove \eqref{asb:eq} for one integral only, that is,
that
\begin{equation}\label{asb1:eq}
\underset{t\notin J}\int \varphi(t) dt
\underset{|t-s|\ge T, s\in J}\int \frac{ds}{|t-s|^2}
 = \bigl(\varphi(s_0) + \varphi(t_0)\bigr) \, |\log(T)| + O(1), \ T\downarrow0,
\end{equation}
for a bounded interval $J = (s_0, t_0)$ with $|s_0-t_0|\ge d_0$. 
Without loss of generality assume that ${J} = (0, 1)$. 
Split the sought integral into the sum 
$X_1 + X_2 + X_3 + X_4$, with
\begin{align*}
X_1 := &\ \int_{1+T}^\infty \varphi(t)dt\int_0^1 
\frac{ds}{(t-s)^2} , 
\\[0.2cm]
X_2 := & \int_{-\infty}^{-T} \varphi(t)dt
\int_0^1 \frac{ds}{(t-s)^2} ,
\\[0.2cm]
X_3 := &\ \int_1^{1+T} \varphi(t)dt\int_0^{1-T} 
\frac{ds}{(t-s)^2} 
+ \int_1^{1+T} \varphi(t)dt
\underset{|s-t|>T, 1-T<s<1} \int \frac{ds}{(t-s)^2} ,\notag\\[0.2cm]
X_4 := &\ \int_{-T}^{0}\varphi(t) dt \int_T^1 
\frac{ds}{(t-s)^2} 
+ \int_{-T}^{0} \varphi(t) dt\underset{|s-t|>T, 0<s<T}
\int \frac{ds}{(t-s)^2} . \notag 
\end{align*}
Direct calculations show that $X_3 + X_4\le C$ 
uniformly in $T\in (0, T_0]$. 
The integral $X_1$ differs from
\begin{equation*}
X_1' = \varphi(1) \int_{1+T}^2 dt \int_0^1 \frac{ds}{(t-s)^2} 
\end{equation*}
at most by a constant independent of $T$. An elementary calculation shows that
\begin{equation*}
X_1' = \varphi(1)|\log(T)| + O(1).
\end{equation*}
Thus $X_1$ satisfies the same formula. 
In the same way one proves the appropriate formula for $X_2$. 
This leads to \eqref{asb1:eq}, and hence to \eqref{asb:eq}. 

The bound \eqref{log_bound:eq} is proved in a similar way by estimating 
integrals of the same type as in the first part of the proof. We omit the details.
\end{proof}

\begin{proof}[Proof of Theorem \ref{CBTheorem:thm}] 
Writing
\begin{equation*}
\CB(a; f) = \CB^{(1)}(a; f) 
+ \bigl(\CB^{(2)}(a; f) - \CB^{(2)}(\chi_\Om; f)\bigr)
+ \CB^{(2)}(\chi_\Om; f),
\end{equation*}
and combining \eqref{sub2:eq}, \eqref{b2diff:eq} and formula 
\eqref{asb:eq} with $\varphi=1$, we obtain the claimed asymptotics \eqref{CBas:eq}. 
\end{proof}

\section{Entanglement entropy and local entropy}
\label{entropy:sect}

In this section we keep using the Fermi symbol $a=a_{T, \mu}$ as in 
\eqref{positiveT:eq} and investigate the special case of the function $f$ given by the 
\textit{$\g$-R\'enyi entropy function} $\eta_\g: \R\mapsto [0, \log(2)]$ 
defined for all $\g >0$ as follows. If $\g\not = 1$, then
\begin{equation}\label{eta_gamma:eq}
\eta_\gamma(t) := \left\{\begin{array}{ll} 
\frac{1}{1-\gamma} \log\big[t^\gamma + (1-t)^\gamma\big]& 
\mbox{ for }t\in(0,1),\\[0.2cm]
0&\mbox{ for }t\not\in(0,1),\end{array}\right.
\end{equation}
and for $\gamma=1$ (the von Neumann case) it is defined as the limit 
\begin{equation}\label{eta1:eq}
\eta_1(t) := \lim_{\gamma\to1} \eta_\gamma(t) = \left\{\begin{array}{ll} -t \log(t) -(1-t)\log(1-t)& \mbox{ for } t\in(0,1),\\[0.2cm]
0& \mbox{ for }t\not\in(0,1).\end{array}\right.
\end{equation} 
From now one we assume that the region $\L$ and the Hamiltonian 
$h$ satisfy Conditions \ref{domain:cond} and \ref{h:cond} 
respectively. The operator $D_\a(\ \cdot\ )$ is as defined in \eqref{Dalpha:eq} 
and the notation $\Om$ is used for the Fermi sea, see Condition \ref{h:cond}. 

If $\L$ is bounded, then \textit{the local (thermal) $\g$-R\'enyi entropy} 
of the equilibrium state at temperature $T>0$ and chemical potential 
$\mu\in\R$ is defined as 
\begin{align}\label{thermal:eq}
\mathrm{S}_\g(T, \mu; \L):=\tr\bigl[
\eta_\g(W_1(a_{T, \mu}; \L))
\bigr],
\end{align}
see for example \cite{HLS}. If one lifts the condition of boundedness, then 
the above quantity may be infinite, but 
\textit{the $\g$-R\'enyi entanglement entropy (EE)} with respect 
to the bipartition 
$\R^d = \L \cup (\R^d\setminus\L)$, defined as 
\bee \label{def:EE}
\mathrm{H}_\g(T,\mu; \L) 
:= \tr D_1(a_{T,\mu},\L;\eta_\g) 
+ \tr D_1(a_{T,\mu},\R^d\setminus\L;\eta_\g),
\ene
is finite, as the next theorem shows. 
Note that these definitions also make sense 
for $T=0$, if one adopts the notation 
$a_{0, \mu} := \underset{T\downarrow 0}\lim \, a_{T,\mu} = \chi_\Om$.  
A somewhat surprising fact is that for bounded $\L$,
\begin{equation}\label{T0:eq}
\mathrm{H}_\g(0, \mu; \L) = 2\,\mathrm{S}_\g(0, \mu; \L).
\end{equation}
As explained in \cite{LeSpSo}, this  
is a consequence of the following two identities: 
$\tr \eta_\g(\chi_\L P_{\Om} {\chi_\L})$ 
$= \tr \eta_\g(P_\Om\chi_\L P_{\Om})$, where  
$P_\Om = \op_1(\chi_\Om)$, and 
$\eta_\g(P_\Om\chi_\L P_{\Om}) 
= \eta_\g(P_{\Om}\chi_{\L^c}P_{\Om})$.
The first identity holds since the non-zero spectra of 
$\chi_\L P_{\Om} {\chi_\L}$ and 
$P_{\Om}\chi_\L P_{\Om}$ coincide. The second one 
follows from the symmetry of $\eta_\g$, that is, from 
the equality $\eta_\g(t) = \eta_\g(1-t),\ t\in [0,1]$. 
 
We are interested in the behaviour of 
the above quantities when $\L$ is replaced with $\a\L$, 
with a large scaling parameter $\a$.  
While the case $T=0$ was investigated in detail 
in \cite{LeSpSo}, in the current paper we concentrate 
on the case $T>0$ and the limit $T\downarrow0$. 
The next theorem shows that the entropies \eqref{thermal:eq} 
and \eqref{def:EE} are both finite, 
and establishes sharp bounds when $\a$ and $T$ both vary 
within certain limits. 

\begin{thm}\label{EE_bound:thm}
Let $d\ge 1$. Suppose that $\a T\ge \a_0$ and $T\in (0, T_0]$ 
with some $\a_0>0, T_0 >0$. Then the $\g$-R\'enyi entanglement entropy
satisfies
\bee \label{EE bound}
|\mathrm{H}_\g(T,\mu;\a\L)| 
\le C \a^{d-1} \big(|\log(T)|+1\big).
\ene
If $\L$ is bounded, then the local $\g$-R\'enyi entropy satisfies
\begin{align}\label{sgam:eq}
\bigl|\mathrm{S}_\g(T, \mu; \a\L) 
- \a^d s_\g(T, \mu)|\L|\bigr|\le C \a^{d-1}(|\log(T)| + 1),
\end{align}
where 
\begin{align*}
s_\g (T, \mu) := \frac{1}{(2\pi)^d} 
\int \eta_\g(a_{T, \mu}(\bxi)) d\bxi.
\end{align*}
The constants in \eqref{EE bound} and \eqref{sgam:eq} 
are independent of $\a$ and $T$, 
but may depend on the parameters 
$\a_0, T_0$, $\mu$, the function $h$ and the region $\L$. 
\end{thm}

The coefficient $s_\gamma(T,\mu)$ 
is called the \textit{$\g$-R\'enyi entropy density} 
(cf.~\cite{LeSpSo2}). It can be expressed in the form,
\begin{equation}\label{entropy density}
s_\g(T,\mu) = 
\begin{cases}
\dfrac{\g}{(\g-1)T} \big(p(T,\mu)-p(T/\g,\mu)\big),\ \mbox{ if }
\g\not = 1,\\[0.4cm]
\dfrac{\partial p}{\partial T}(T,\mu),\ \mbox{ if }\g = 1,
\end{cases}
\end{equation}
in terms of the \textit{pressure}
$$ 
p(T,\mu) := \int\frac{\mathcal N(E)}{1+e^{(E-\mu)/T}} \,dE,
$$
and the \textit{integrated density of states}
$$ 
\mathcal N(E) := \frac1{(2\pi)^{d}} \int\chi_{[0,\infty)}(E-h(\bxi)) \,d\bxi\,,\; E\in \mathbb R,
$$
of the free Fermi gas. The relation \eqref{entropy density} for $\gamma =1$ is a standard 
thermodynamic relation, 
see for instance \cite{Bal}.

For $d=1$, apart from the bounds, we can also determine 
the asymptotic behaviour of the local (or thermal) entropy and of the EE. 

\begin{thm}\label{EE:thm}
Let $d = 1$, and let $I\subset\R$
be given by \eqref{Kint:eq}. 
Then the EE satisfies
\begin{align}\label{EEas:eq}
\mathrm{H}_\g(T, \mu; \a I) = 2\om \CB(a_{T, \mu}, \eta_\g) + 
o(|\log(T)|+1), 
\end{align}
and if $I_0 = I_{K+1} = \varnothing$, then with $s_\g$ from \eqref{entropy density} 
the local entropy satisfies
\begin{align}\label{Sgamas:eq}
\mathrm{S}_\g(T, \mu; \a I) = \a 
s_\g(T, \mu) |I| + 2K \CB(a_{T, \mu}; \eta_\g) + o(|\log(T)|+1),
\end{align}
as $\a T\ge \a_0$, $\a\to\infty$.
\end{thm}

A proof of the leading large-scale behaviour of the local entropy 
$\mathrm{S}_\g(T,\mu;\a\L)$ 
at fixed $T>0$ appeared (among other things) first in \cite{PS,BHK} (for $\g=1$). The 
sub-leading correction in 
dimension $d=1$ in \eqref{Sgamas:eq} is new. 
The extension to dimension $d\ge2$ is subject of \cite{Sob-future}.

If in Theorem \ref{EE:thm} we also assume that $T\downarrow 0$, then 
the asymptotic formulas take a more explicit form. To state this 
result recall that due to Condition \ref{h:cond}, the Fermi sea 
has the form \eqref{omj:eq}, with a finite $N\in \mathbb N$. 

\begin{cor}\label{tostep:cor}
Let $d=1$, $I$ as in \eqref{Kint:eq}, and let $N$ be the number of connected components 
of the Fermi sea, 
see \eqref{omj:eq}. Let $\a T\ge \a_0$ and $T\downarrow 0$. Then the EE satisfies
\begin{align}\label{tostep:eq}
\mathrm{H}_\g(T, \mu; \a I) 
= \om N  \,\frac{1+\g}{6\g}|\log(T)| + o(|\log(T)|+1).
\end{align}
\end{cor}

It is worth pointing out that the coefficient in front of 
$|\log (T)|$ agrees with the asymptotic coefficient found 
in \cite{LeSpSo} for the zero temperature case. 
Indeed, with the notation that we presently use, 
the main Theorem of \cite{LeSpSo} states that 
for a bounded $I$, that is, with $\om = 2K$, we have (see \eqref{T0:eq})
\begin{align*}
\mathrm{H}_\g(0,\mu; \a I) \; = \; &2\,\mathrm{S}_\g(0,\mu; \a I)  
\\
=\; & \omega N \,\frac{1+\gamma}{6\g}\, \log(\a) + o(\log(\a)),\  
\a\to\infty.
\end{align*}
Clearly, the coefficient in this formula is the same as 
in Corollary \ref{tostep:cor}.
Therefore, if we identify $\a$ inside the logarithm with $1/T$ we recover the above asymptotic 
expansion in Corollary \ref{tostep:cor}.

\begin{proof}[Proof of Theorem \ref{EE_bound:thm}]
It is easy to see that 
\bee \label{def1:EE}
\mathrm{H}_\g(T,\mu; \a\L) 
= \tr D_\a(a_{T,\mu},\L;\eta_\g) 
+ \tr D_\a(a_{T,\mu},\R^d\setminus\L;\eta_\g).
\ene
Now, let $\phi\in \plainC{\infty}(\R)$ be such that $0\le\phi\le1$ and
\begin{equation*}
 \phi(t) = \left\{\begin{array}{ll} 1& 
 \mbox{ for }t\le 1/4\\0& \mbox{ for }t\ge3/4\end{array}\right..
\end{equation*}
If $\g\not =1$, then $\eta_\gamma\phi$ and 
$\eta_\gamma(1-\phi)$ satisfy 
Condition \ref{f:cond} with 
$t_0=0$ and $t_0=1$, respectively, 
and with $\varkappa = \min\{1, \g\}$. 
The functions $\eta_1\phi$ and $\eta_1(1-\phi)$ 
satisfy Condition \ref{f:cond} with arbitrary $\g <1$.
Since the mapping $f\mapsto D_\a(a,\L; f)$ is linear,
we have
\begin{equation}\label{splitEE:eq}
D_\a(a,\L; \eta_\g) 
= D_\a(a,\L; \eta_\g\phi) + D_\a(a,\L; \eta_\g(1-\phi)).
\end{equation}
Applying Theorem \ref{entropy:thm} with $R=1$ to each term on the 
right-hand side, we conclude 
for $\a T\ge\a_0$ and $0 < T\le T_0$,  that 
\begin{align}\label{Dada:eq}
\| D_\a(a_{T, \mu}, \L; \eta_\g)\|_{\GS_1}
+ \| D_\a(a_{T, \mu}, \R^d\setminus\L; \eta_\g)\|_{\GS_1}
\le C\a^{d-1} \bigl(|\log(T)| + 1\bigr).
\end{align}
In view of \eqref{def1:EE}, this leads to \eqref{EE bound}.

In order to prove \eqref{sgam:eq}, we rewrite \eqref{thermal:eq}:
\begin{align}\label{S_gamma}
\mathrm{S}_\g(T,\mu;\L) 
= \tr[\chi_\L \eta_\g(\op_\a(a_{T,\mu}))\chi_\L] 
+ \tr D_\a(a_{T,\mu},\L; \eta_\g).
\end{align}
For the first trace we have the simple identity
\begin{equation*}
\tr[\chi_\L \eta_\g(\op_\a(a_{T,\mu}))\chi_\L] 
= \tr[\chi_\L \op_\a(\eta_\g(a_{T,\mu}))\chi_\L] 
=  \a^d s_\g(T,\mu) \,|\L| \,.
\end{equation*}
Together with the bound \eqref{Dada:eq} for the second trace, 
this yields \eqref{sgam:eq}.
\end{proof}

\begin{proof}[Proof of Theorem \ref{EE:thm}]
Applying Theorem \ref{interval_lowT:thm} 
to each term on the right-hand side of \eqref{splitEE:eq}, 
and using \eqref{compl:eq}, 
we obtain as $\a T\ge \a_0$, $\a\to\infty$, that
\begin{align*}
\tr D_\a(a_{T, \mu}, I; \eta_\g\phi)
= \omega\CB(a_{T, \mu}; \eta_\g\phi)  + o(|\log(T)|+1),
\end{align*}
and 
\begin{align*}
\tr D_\a(a_{T, \mu}, I^c; \eta_\g\phi)
= \omega\CB(a_{T, \mu}; \eta_\g\phi)  + o(|\log(T)|+1),
\end{align*}
where $\om=\om(I)=\om(I^{c})$. Similar formulas can be written 
with the function $(1-\phi)\eta_\g$ as well. Thus, remembering the 
linearity of the map $f\mapsto \CB(a; f)$ 
(see definition \eqref{cb:eq}), we obtain  \eqref{EEas:eq}. 

The asymptotics in \eqref{Sgamas:eq} is obtained in the same 
way using \eqref{S_gamma} and \eqref{weyl:eq}. 
\end{proof}

\begin{proof}[Proof of Corollary \ref{tostep:cor}]
The claimed formula immediately follows from \eqref{EEas:eq} and 
the asymptotic relation \eqref{CBas:eq} after observing that (cf.~\cite{LeSpSo})
$$ 
U(1, 0; \eta_\g) = \int_0^1 \frac{\eta_\g(t)}{t(1-t)} dt = {\pi^2}\,\frac{1+\g}{6\g}.
$$
\end{proof}

One should note that in the same way one could replace the coefficient 
$\CB(a_{T, \mu}, \eta_\g)$ by its asymptotics \eqref{CBas:eq} in 
formula \eqref{Sgamas:eq} as well. However, the specific entropy 
density $s_\g(T, \mu)$ in the leading term would 
also need to be expanded in $T\downarrow 0$, and 
the precise place of the $\CB(\ \cdot\ )$-term in the resulting 
expansion of $\mathrm{S}_\g$ will depend on the relationship 
between $\a T$ and $T$. We do not go into these details.

\begin{appendix}

\section{The Helffer--Sj\"ostrand formula} 
When studying functions of self-adjoint operators 
we rely on 
the Helffer--Sj\"ostrand formula which 
holds for arbitrary operators $A=A^*$  
and arbitrary smooth functions $f\in\plainC{n}_0(\R), n\ge 2$ 
($z:=x+iy, 
\frac{\partial}{\partial\bar{z}} 
:= \frac12 (\frac{\partial}{\partial x} + i\frac{\partial}{\partial y})$):
\begin{equation}\label{HS:eq}
	f(A) = \frac{1}{\pi} \iint \frac{\p}{\p \bar z} \tilde f(x, y)\, 
	(A-x-iy)^{-1} dx dy,
\end{equation}
where $\tilde f =  \tilde f(x, y)$ 
is an almost analytic extension of the function $f$, see 
\cite[Chapter 2]{EBD}. An almost analytic extension of 
$f\in \plainC{n}(\R)$ is a $\plainC1(\R^2)$-function $\tilde f$, such that 
$f(x) = \tilde f(x, 0)$ 
and $\big|\frac{\p}{\p \bar z} \tilde f(x, y)\big|\le C|y|$.  
For the sake of brevity we use the representation \eqref{HS:eq} for compactly supported 
functions only, so that the integral \eqref{HS:eq} is norm-convergent. 

Let us describe a convenient almost analytic extension of a function 
$f\in\plainC{n}_0(\R)$. For an arbitrary $r >0$ introduce the function 
\begin{equation*}
U_r(x, y) := 
\begin{cases}
1,\ |y|< \lu x\ru_r,\\[0.2cm]
0,\  |y|\ge \lu x\ru_r,
\end{cases}
\lu x\ru_r := \sqrt{x^2 + r^2}. 
\end{equation*}
Later, we need a function $\z\in\plainC\infty_0(\R)$ be a function 
such that 
\begin{equation}\label{zeta:eq}
\z(t) = 1 \ \ \textup{for}\ \  |t|\le 1/2, \ \ 
\textup{and}\ \  \z(t) = 0 \ \ \textup{for}\ \ |t|\ge 1.
\end{equation}

\begin{lem}\label{ab:lem} 
	Let $f\in \plainC{n}(\R), n\ge 2$. Then for any $r >0$ the function $f$ 
	has an almost analytic extension 
	$\tilde f = \tilde f(\ \cdot\ , \ \cdot\ ; r)\in\plainC1(\R^2)$ such that 
	$\tilde f(x, y; r) = 0$ if $|y|>\lu x\ru_r$. Moreover, the derivative,
     $\frac{\p}{\p \overline{z}}\tilde f(x, y; r)$,
	satisfies the bound 
	\begin{equation}\label{ab:eq}
	\Big|\frac{\p}{\p \overline{z}}\tilde f(x, y; r)\Big|
	\le C_n F(x; r)
	 |y|^{n-1} U_r(x, y),
	\end{equation} 
 where 
\begin{equation*}
F(x; r) := \sum_{l=0}^{n} |f^{(l)}(x)|\lu x\ru_r^{-n+l}.
\end{equation*}	
The constant $C_n$ does not depend on $f$ or the constant $r$. 
\end{lem}

The proof of this lemma is a marginal modification of the proof contained in \cite[Chapter 2]{EBD} 
and is thus omitted.  
 
\end{appendix}

\bibliographystyle{amsplain}

\end{document}